\documentclass[11pt,a4paper]{amsart}
\usepackage{amsmath,amssymb,amscd,amsthm,amsxtra,array,soul,ulem}

\usepackage[latin9]{inputenc}

\usepackage[usenames,dvipsnames]{xcolor}
\usepackage[
colorlinks=true,
linkcolor=blue,
citecolor=blue,
urlcolor=blue,
]{hyperref}

\usepackage{ulem}

\newtheorem{theorem}{Theorem}[section]
\newtheorem{lemma}[theorem]{Lemma}
\newtheorem{proposition}[theorem]{Proposition}
\newtheorem{corollary}[theorem]{Corollary}

\newtheorem{conjecture}[theorem]{Conjecture}

\theoremstyle{definition}
\newtheorem{definition}[theorem]{Definition}

\theoremstyle{remark}
\newtheorem{remark}[theorem]{Remark}
\numberwithin{equation}{section}
\DeclareMathOperator*{\esssup}{ess\,sup}

\newcommand{\absval}[1]{\mbox{$|#1|$}}

\def\R{{\mathbb R}}

\newcommand{\norm}[1]{\mbox{$\left\| #1 \right\|$}}

\def\Xint#1{\mathchoice
  {\XXint\displaystyle\textstyle{#1}}%
  {\XXint\textstyle\scriptstyle{#1}}%
  {\XXint\scriptstyle\scriptscriptstyle{#1}}%
  {\XXint\scriptscriptstyle\scriptscriptstyle{#1}}%
  \!\int}
\def\XXint#1#2#3{{\setbox0=\hbox{$#1{#2#3}{\int}$}
    \vcenter{\hbox{$#2#3$}}\kern-.5\wd0}}

\def\avgint{\Xint-}
\newcommand{\vertiii}[1]{{\left\vert\kern-0.25ex\left\vert\kern-0.25ex\left\vert #1 
    \right\vert\kern-0.25ex\right\vert\kern-0.25ex\right\vert}}

\numberwithin{equation}{section}

\begin{document}

\title[ Weighted Poincar\'e]{  Degenerate Poincar\'e-Sobolev inequalities }

\author{Carlos P\'erez}
\address[Carlos P\'erez]{ Department of Mathematics, University of the Basque Country, IKERBASQUE 
(Basque Foundation for Science) and
BCAM \textendash  Basque Center for Applied Mathematics, Bilbao, Spain}
\email{carlos.perezmo@ehu.es}

\author{Ezequiel Rela}
\address[Ezequiel Rela]{Department of Mathematics,
Facultad de Ciencias Exactas y Naturales,
University of Buenos Aires, Ciudad Universitaria
Pabell\'on I, Buenos Aires 1428 Capital Federal Argentina} \email{erela@dm.uba.ar}

\thanks{C. P. is supported by grant  MTM2017-82160-C2-1-P of the Ministerio de Econom\'ia y Competitividad (Spain), grant IT-641-13 of the Basque Government, and IKERBASQUE}
\thanks{E.R. is partially supported by grants UBACyT 20020130100403BA, PIP (CONICET) 11220110101018, by the Basque Government through the BERC 2014-2017 program, and by the Spanish Ministry of Economy and Competitiveness MINECO: BCAM Severo Ochoa accreditation SEV-2013-0323.}



\begin{abstract}
We study weighted Poincar\'e and Poincar\'e-Sobolev type  inequalities with an explicit analysis on the dependence on the $A_p$ constants of the involved weights. We obtain inequalities of the form
\begin{equation*}
\left (\frac{1}{w(Q)}\int_Q|f-f_Q|^{q}w\right )^\frac{1}{q}\le C_w\ell(Q)\left (\frac{1}{w(Q)}\int_Q |\nabla f|^p w\right )^\frac{1}{p},
\end{equation*}
with different quantitative estimates for both the exponent $q$ and the constant $C_w$.
We will derive those estimates together with a large variety of related results as a consequence of a general selfimproving property shared by functions 
satisfying the inequality
$$
\avgint_Q |f-f_Q| d\mu \le a(Q),
$$
for all cubes $Q\subset\mathbb{R}^n$ and where $a$ is some functional that obeys a specific discrete geometrical summability condition.  We introduce a Sobolev-type exponent $p^*_w>p$ associated to the weight $w$ and obtain further improvements involving $L^{p^*_w}$ norms on the left hand side of the inequality above. For the endpoint case of $A_1$ weights we reach the classical critical Sobolev exponent $p^*=\frac{pn}{n-p}$ which is the largest possible and provide different type of quantitative estimates for $C_w$. We also show that this best possible estimate cannot hold with an exponent on the $A_1$ constant smaller than $1/p$.  

As a consequence of our results (and the method of proof) we obtain further extensions to two weights Poincar\'e inequalities and to the case of higher order derivatives. Some other related results in the spirit of the work of Keith and Zhong on the open ended condition of Poincar\'e inequality are obtained using extrapolation methods.  We also apply our method to obtain similar estimates in the scale of Lorentz spaces.

We also provide an argument based on extrapolation ideas showing that there is no $(p,p)$, $p\geq1$, Poincar\'e inequality valid for the whole class of $RH_\infty$ weights by showing their intimate connection with the failure of Poincar\'e inequalities, $(p,p)$ in the range $0<p<1$.

\end{abstract}

\maketitle

\begin{center}
\today 
\end{center}
\normalem

\section{Introduction and Main Results }

The celebrated Moser iteration method (see for instance \cite{HKM,Saloff-Coste-LMS-LN}) is a powerful and flexible devise to prove  the local H\"older regularity of the weak solutions of elliptic PDE due independently, and by different methods, to De Giorgi and Nash.  This method has  two  important key steps.  One is the $(2,2)$ Poincar\'e inequality and the other is its correspondent Poincar\'e-Sobolev  $(2^*,2)$ inequality where $2^*$ is the classical Sobolev exponent.   In \cite{FKS} it is considered this problem within the context of degenerate elliptic PDE, namely it is considered the 
operator $Lu= \text{div}(A(x) \nabla u)$ where $A$ is an $n \times n$ real symmetric matrix in $\mathbb{R}^n$ satisfying the ``degenerate" elliptic condition 
 $$  A(x)\xi.\xi  \approx |\xi|^2 w(x), $$
where the ``degeneracy" is given by a weight $w$ in the $A_2$ class. To do this it is proven in \cite{FKS} appropriate weighted Poincar\'e and Poincar\'e-Sobolev inequalities (cf. also \cite{HKM}).  One of the main purposes of this paper is to improve some of the main results from \cite{FKS}.  To be more precise we are interested in proving \emph{weighted} Poincar\'e-Sobolev inequalities of the form
\begin{equation*}
\left (\frac{1}{w(Q)}\int_Q|f-f_Q|^{q}w\right )^\frac{1}{q}\le C_w\ell(Q)\left (\frac{1}{w(Q)}\int_Q |\nabla f|^p w\right )^\frac{1}{p},
\end{equation*}
where  $1\leq p\leq q\le \infty$ and $w$ is a weight function in the $A_p$ class of Muckenhoupt (see Section \ref{sec:prelim} for the precise definitions). We will improve these results in two ways: 
1) By producing a quantitative  control of the constant  $C_w$ and 2) by producing a more precise control of the exponent $q$ as a function of $p,n$ and, often, the $A_p$ constant of the weight. To accomplish this  we will follow the general framework introduced in  \cite{FPW98} which allowed to produce in a unified manner the main results of  \cite{FKS} and many others. These results are obtained in the context of Spaces of Homogeneous Type where the underlying measure is doubling. The main results were improved  in \cite{MacManus-Perez-98}  and were also further exploited in the context of nondoubling meausures in \cite{OP-nondoubling}.  In the current  work we introduce new techniques to continue using this method with some novelties which allow to sharpen and improve the main results from the articles mentioned above.  As a consequence, we will also show that our new method contains a different proof of the John-Nirenberg estimate for generalized BMO functions in a more precise way.

Consider as a starting point an inequality of the form

\begin{equation}
\avgint_Q |f-f_Q| d\mu \le a(Q),
\end{equation}
where $a:\mathcal{Q}\to (0,\infty)$ is a general functional defined over the family of cubes in $\mathbb{R}^n$ with sides parallell to the coordinate axes, denoted by $\mathcal{Q}$,   satisfying an appropriate  extra discrete summability geometric  condition. Then there is a self-improving  phenomenon of the inequality above that allows to obtain an $L^r$ estimate for some $r>1$ depending on the choice of the functional $a$. The first result of this type was obtained by L. Saloff-Coste in \cite{Saloff-Coste-IMRN} and later on, in a the context of metric spaces with a doubling measure $(X,d,\mu)$, was obtained in \cite{HaK1} by P. Hajlasz and P. Koskela for the more standard situation, 
$$
a(B)=   \frac{r(B)}{ \mu(B)}\int_B  g d\mu.
$$
We refer to  \cite{HaK2} for a general account of the relevance of Poincar\'e type inequalities in such  general contexts.

A different and more flexible approach was introduced in \cite{FPW98}. The key point of this paper is the use of the following geometric type hypothesis that recalls  Carleson's condition.

\begin{definition} \label{def:Dp}
Let $w$ be any weight. We say that the functional $a$ satisfies the weighted $D_p$ condition for $0<p<\infty$ if there is a constant $C$ such that for any cube $Q$ and any family $\Lambda$ of pairwise disjoint subcubes of $Q$, the following inequality holds:
\begin{equation}\label{eq:Dp}
\sum_{P\in\Lambda}a(P)^pw(P)\le C^p a(Q)^pw(Q).
\end{equation}
The best possible constant $C$ above is denoted by $\|a\|$ and also we will write in this case that $a\in D_p(w)$.
\end{definition}

We include here the main theorem from \cite{FPW98} in the simplest context  which is the initial result in this theory. We will use the following notation for the  weak weighted normalized norm over a cube $Q\subset \mathbb{R}^n$.
\begin{equation*}
\|f\|_{L^{p,\infty}_{Q, \frac{wdx}{w(Q)} }}:=\sup_{\lambda>0}\lambda \left (\frac{w(\{x\in Q: |f(x)|>\lambda\})}{w(Q)}\right )^{1/p}.
\end{equation*}

\begin{theorem}[\cite{FPW98}]\label{thm:general-a(Q)}
Let $w$ be an $A_\infty$ weight and let $a$ be a functional satisfying the weighted $D_p$ condition \eqref{eq:Dp} for some $p>0$. Let $f$ be a locally integrable function such that 
\begin{equation*}
\frac{1}{|Q|}\int_Q|f-f_Q|\le  a(Q).
\end{equation*}
Then the following weak type inequality holds:
\begin{equation}\label{eq:weak-a(Q)}
\|f-f_Q\|_{L^{p,\infty}_{Q,\frac{wdx}{w(Q)}}}\le C\|a\| a(Q).
\end{equation}
\end{theorem}

Although this result is very flexible and useful as can be shown in \cite{Saari} or \cite{PPSS},   the method does not provide a good control of the bound $C$ and $r$ specially in the weighted situation.  We also refer to the interesting paper \cite{BKM} where a connection with the  so called John-Nirenberg spaces can be found. A model example of a functional $a$ is the one defined as
\begin{equation}\label{eq:a} 
a(Q,f):=  \ell(Q)\,            \left(\frac{1}{w(Q)}\int_Q |\nabla f(x)|^p\,wdx \right)^{\frac1p}.
\end{equation}
Combining Theorem \ref{thm:general-a(Q)} with the truncation method, also called \emph{weak implies strong} (see Section \ref{truncation}), we can prove that 
\begin{equation}
\left (\frac{1}{w(Q)}\int_Q|f-f_Q|^{p}w\right )^\frac{1}{p}
\leq Ca(Q,f),
\end{equation}
 with $w\in A_p$. However, this method is not so precise since we loose control of the $A_p$ constant $[w]_{A_p}$.  
We use a different general approach which allows to prove directly
\begin{equation}\label{eq:2weightApPI p-p}
\left( \int_Q|f- f_{Q}|^{p}u\right )^\frac{1}{p}\leq c_{n} [u,v]^{1/p}_{A_p}\, \ell(Q)  \left (\int_Q |\nabla f|^p v\right )^\frac{1}{p},
\end{equation}
as a consequence of Theorem \ref{thm:Lp(w)-a(Q)-clean} below. This estimate is not so well known and can be obtained by standard methods using fractional and maximal operators, combined with the truncation method. Also, using a variant of  the proof of of Theorem \ref{thm:Lp(w)-a(Q)-clean}  we can consider generalized Poincar\'e type inequalities related to higher order derivatives. Indeed, we will show the following two weighted estimate in Corollary \ref{cor:Poincare(pp)-higher} as a consequence of Theorem \ref{thm:higher-a(Q)}:
\begin{equation}\label{eq:Poincare(pp)-higher}
\left( \int_Q|f- P_{Q}f |^{p}u\right )^\frac{1}{p}\leq c_{m,n} [u,v]^{1/p}_{A_p}\, \ell(Q)^m  \left (\int_Q |\nabla^m f|^p v\right )^\frac{1}{p},
\end{equation}
where $P_{Q}f$ is an appropriate polynomial of order $m-1$. This estimate seems not be known since the truncation method  cannot be used in the case of higher order derivatives.

To do this we will be assuming a more precise geometric hypothesis on the functional $a(Q)$ stated in Definition \ref{def:SDp}.
Here we will impose that  the functional $a$ preserves some sort of ``smallness". This variant of the $D_p$ condition will produce more refined results as those obtained in \cite{FPW98} and subsequent papers  \cite{MacManus-Perez-98, LernerPerez-poin-scandi}.

We start by introducing the notion of  ``smallness'' of a family of pairwise disjoint subcubes of a given cube $Q$.

\begin{definition}\label{def:L-small}
Let $L>1$ and let $Q$ be a cube. We will say that a family of pairwise disjoint subcubes $\{Q_i\}$ of $Q$ is $L$-small if
\begin{equation}\label{eq:L-small}
\sum_{i}|Q_i|\le \frac{|Q|}{L}.
\end{equation}
We will also denote $\{Q_i\}\in S(L)$
\end{definition}

Now, the correct notion of a $D_p$-type  condition in this context is the following.

\begin{definition}\label{def:SDp}
Let $w$ be any weight and let $s>1$. We say that the functional $a$ satisfies the weighted $SD^s_p(w)$ condition for $0\le p<\infty$ if there is a constant $C$ such that for any cube $Q$ and any family $\{Q_i\}$ of pairwise disjoint subcubes of $Q$ such that $\{Q_i\}\in S(L)$, the following inequality holds:
\begin{equation}\label{eq:SDp}
\sum_{i}a(Q_i)^pw(Q_i)\le C^p \left (\frac{1}{L}\right )^{\frac{p}{s}}a(Q)^pw(Q).
\end{equation}
The best possible constant $C$ above is denoted by $\|a\|$ and also we will write in this case that $a\in SD^s_p(w)$. We say in this case that the functional $a$ ``preserves'' the smallness condition of the family of cubes.
\end{definition}

At this point, we should present an example of a functional satisfying the $SD_p^{s}(w)$ condition. The following very general model for $a(Q)$ fulfills the requirements. Let $\mu$ be any Radon measure and define the fractional functional
\begin{equation}\label{eq:general-a(Q)}
a(Q)=\ell(Q)^{\alpha}\left(\frac{1}{w(Q)}\mu(Q) \right)^{1/p}, \qquad \alpha, p>0.
\end{equation}
More specific examples are 
\begin{equation*}\label{eq:gradients-a(Q)}
a(Q)=\ell(Q)\left(\frac{1}{w(Q)}\int_Q g\right)^{1/p},\quad \quad  a(Q)=\ell(Q)^m\left(\frac{1}{w(Q)}\int_Q |\nabla^m f|^p\right)^{1/p},
\end{equation*}
where $m=1,2,\dots$. We will include in Lemma \ref{lem:L-small} a very simple computation showing that the functional \eqref{eq:general-a(Q)}
satisfies the $SD_p^{n/\alpha}(w)$ condition.

\subsection{Generalized \texorpdfstring{$(p,p)$}{(p,p)} Poincar\'e}

Our first main result is the next theorem which is an important improvement of Theorem \ref{thm:general-a(Q)} obtaining the same inequality for the \emph{strong} norm.

\begin{theorem}\label{thm:Lp(w)-a(Q)-clean}

Let $w$ be any $A_\infty$ weight. Consider also the functional   $a$ such that for some $p\ge 1$ it satisfies the weighted condition $SD^s_p(w)$ from \eqref{eq:SDp} with $s>1$ and constant $\|a\|$. 
Let $f$ be a locally integrable function such that
\begin{equation*}
\frac{1}{|Q|}\int_{Q} |f-f_{Q}| \le a(Q),
\end{equation*}
for every cube $Q$. Then, there exists a dimensional constant $C_n$ such that for any cube $Q$
\begin{equation}\label{eq:First main estimate}
\left( \frac{1}{ w(Q)  } \int_{ Q }   |f -f_{Q}|^p     \,wdx\right)^{\frac{1}p}   \leq  C_n\, s \|a\|^s  a(Q).
\end{equation}

\end{theorem}

\begin{remark}
An important point is that the $A_{\infty}$ constant does not appear in the result, estimate \eqref{eq:First main estimate}, even though it is 
assumed that the weight $w$ is an $A_\infty$ weight. We believe that the result holds without assuming any condition on the weight $w$.
\end{remark}

Note that obtaining the strong inequality is relevant,  since there is no need to use the truncation method and then,  many other  functionals beyond the case of the gradient can be considered. As an example, we can derive the $(p,p)$ Poincar\'e inequality \eqref{eq:2weightApPI p-p} directly, in a different way than the one presented in Proposition \ref{pro:two-weights} which requires fractional and maximal operators.

We now present several corollaries of this first main theorem.

\begin{corollary}\label{cor:Poincare(p,p)-for measures}
Let $p>1$ and $\alpha>0$. Let $\mu$ be  a measure and let $w$ be an $A_\infty$ weight. Suppose that $f$ is a locally integrable function such that for some constant $a>0$, we have 
\begin{equation*}
\frac{1}{|Q|}\int_{Q} |f-f_{Q}| \leq a\, \ell(Q)^{\alpha}\left(\frac{1}{w(Q)}\mu(Q) \right)^{1/p},
\end{equation*}
for every cube $Q$. Then 
$$
\left (\frac{1}{w(Q)}\int_Q |f-f_Q|^p\, w\ dx\right )^{1/p}\le \frac{c_n a}{\alpha} \ell(Q)^{\alpha}\left(\frac{1}{w(Q)}\mu(Q) \right)^{1/p},
$$
where $c_n$ is a dimensional constant.
\end{corollary}

As a consequence we derive the following  two weight Poincar\'e inequality where we don't use the truncation method.

\begin{corollary}\label{cor:Poincare(p,p)-twoweight}
Let $(u,v)\in A_p$,  $u\in A_\infty$. Then the following Poincar\'e $(p,p)$ inequality holds
$$
\left (\frac{1}{u(Q)}\int_Q |f-f_Q|^p\, u\ dx\right )^{1/p}\le c_n [u,v]^{\frac{1}{p}}_{A_p}\ell(Q)\left (\frac{1}{u(Q)}\int_Q|\nabla f|^{p} \, v \ dx\right )^{1/p},
$$
where $c_n$ is a dimensional constant.
\end{corollary}

The proof of Theorem \ref{thm:Lp(w)-a(Q)-clean} was inspired by the beautiful argument used by Journ\'e in \cite{Journe} to prove John-Nirenberg's theorem. In fact, from our Theorem \ref{thm:Lp(w)-a(Q)-clean} we can derive the following corollary, which easily implies the celebrated John-Nirenberg's inequality with a different argument.

\begin{corollary}\label{cor:pre-JN}
Let $f$ be a locally integrable function and suppose that there is an {\bf increasing} functional $a$ such that 
\begin{equation*}
\frac{1}{|Q|}\int_{Q} |f-f_{Q}| \leq a(Q),
\end{equation*}
for every cube $Q$. Then, there exists a dimensional constant $c_n$ such that for any cube $Q$
\begin{equation*}
\norm{f -f_{Q}}_{\exp L(Q, \frac{dx}{|Q|})} \leq c_n\,a(Q)
\end{equation*}
\end{corollary}
 Here we used the usual Orlicz type norm: 
$$
\norm{g}_{\exp L(Q, \frac{dx}{|Q|})} =\inf \{\lambda>0\,:\, \frac{1}{|Q|}\int_Q\Phi \left (\frac{g(x)}{\lambda }\right )\,dx\leq 1 \}    \qquad g \geq 0
$$
with $\Phi(t)=e^t -1       $

For the proof we observe readily that if $a$ is increasing, namely  $P\subset Q$ implies $a(P)\leq a(Q)$, then $a$ satisfies the unweighted $SD^p_p$ 
for any $p>0$ with $\|a\| \leq 1$. In particular if $1<p<\infty$, by \eqref{eq:First main estimate} in Theorem \ref{thm:Lp(w)-a(Q)-clean}, there exists a dimensional constant $c_n$ such that for any cube $Q$,
\begin{equation*}
\left( \frac{1}{ |Q|  } \int_{ Q }   |f -f_{Q}|^p     \,dx\right)^{\frac{1}p}   \leq  c_n \,p\,  a(Q). 
\end{equation*}
Now, we use the following well know estimate: let  $(X,\mu)$ be a probability space and let a function $g$  such that for some $p_0\geq 1$, $c>0$, and $\alpha>0$ we have that 
$$
\norm{g}_{L^p(X,\mu)} \leq c\,p^{\alpha} \qquad p\geq p_0.
$$
Then for a universal multiple of $c$,  
$$
\norm{g}_{\exp (L)^{\frac{1}{\alpha}},(X,\mu)} \leq c.\,
$$
We conclude by considering  $X=Q$, $d\mu=\frac{dx}{|Q|}$,   $g=\frac{|f-f_Q|}{a(Q)}$ and $\alpha=1$. Now, specializing the above corollary for the case $a\equiv 1$, we recover John-Nirenberg's theorem.

\begin{remark}

A variant of the method we use also produce the following weighted estimate 
\begin{equation*}
\norm{f -f_{Q}}_{\exp L(Q, \frac{wdx}{w(Q)} )} \leq C_n\,\|f\|_{BMO}\, [w]_{A_{\infty}}\,a(Q) 
\end{equation*}
when $a$ is an increasing functional.
\end{remark}

As a consequence of Theorem \ref{thm:Lp(w)-a(Q)-clean}, using ideas from extrapolation theory from  \cite{Duo} \cite{CMP-Book}    and we are able to derive in Corollary \ref{cor:K-Z phenomenon} a result in the spirit of the celebrated theorem of Keith and Zhong \cite{KZ} on the open ended property of Poincar\'e inequalities. The proof of this corollary will be presented in Section \ref{sec:KZ}.

\begin{corollary}\label{cor:K-Z phenomenon}
Let $w\in A_{p_0}$ and $1<p_0<\infty$ and let also $\varphi: [1,\infty)\to (0,\infty)$ be non-decreasing. Let  $(f,g)$ be a couple of functions  satisfying  the following Poincar\'e $(1,p_0)$  for any $w\in A_{p_0}$, 
\begin{equation}\label{eq:KZ-hyp}
\frac{1}{|Q|}\int_Q|f-f_Q|dx \leq \varphi([w]_{A_{p_0}})  \ell(Q) \left( \frac{1}{w(Q)} \int_Q g^{p_0}\,wdx \right )^{\frac{1}{p_0}},
\end{equation}

Then, for any $p$ such that $1<p< p_0$ the following estimate holds for any $w\in A_p$: 
\begin{equation*}
\left (\frac{1}{w(Q)}  \int_Q|f-f_{Q}|^{p} \,wdx  \right )^\frac{1}{p}\leq c\,\varphi( c_{p,p_0,n}\,  [w]_{A_{p}}^{\frac{p_{0}-1}{p-1}} )
  \ell(Q) \left(\frac{1}{w(Q)}   \int_Q g^{p}\,wdx  \right)^{\frac{1}{p}} 
\end{equation*}
where $c$ is a constant depending on $p,p_0$ and the dimension.
\end{corollary}

\subsection{Generalized \texorpdfstring{$(p_w^*,p)$}{p*w,p} Poincar\'e-Sobolev}

In this section we want to move further in the direction of Poincar\'e-Sobolev inequalities. In the case of Lebesgue measure and for the particular case of the functional given by the gradient, the critical index is $p^*=\frac{pn}{n-p}>p$ (see \cite{Sobolev,Gagliardo,Morrey}).

Our next results provide further improvements for $(p_{w}^*,p)$ Poincar\'e type inequalities for $A_q$ weights, $1\leq q\leq p$. Here we introduce the Sobolev index $p_{w}^* >p$ to
obtain a wider range of exponents. In addition, we obtain sharper estimates on the dependence on the $A_q$ constants. We will obtain inequalities of the form
\begin{equation*}
\left \| f-f_Q\right \|_{L^{p_{w}^*}{(Q,\frac{wdx}{w(Q)})}}\le C_n \varphi(w)\ell(Q) \left \|\nabla f\right \|_{L^{p}{(Q,\frac{wdx}{w(Q)})}},
\end{equation*}
where the exponent $p_{w}^*$ depends on the weight $w\in A_q$. Note that we fix the value of $p$ on the right hand side and pursue the best possible exponent $p_{w}^*$. There will be some sort of balance between our best $p_{w}^*$ and the sharper quantitative estimate for $\varphi(w)$.  The main difficulty to overcome is to obtain the ``smallness preservation'' for the functional $a$ when dealing with higher exponents. In that case, it will be crucial to use some extra geometric consequences of the membership of the weight into the $A_q$ class. 

We have the following theorem.

\begin{theorem} \label{thm:ptimes-Aq}
Let $1 \leq p < n $ and let $w\in A_q$ with $1\le q\le p$. Let also $p_w^* $ be defined by 
\begin{equation}\label{eq:ptimes-Aq}
\frac{1}{p} -\frac{1}{ p_w^* }=\frac{1}{n(q+\log [w]_{A_q})}.
\end{equation}

Let $a$ be the functional defined by
$$a(Q)=\ell(Q)\left(\frac{1}{w(Q)}\mu(Q) \right)^{1/p},
$$
where $\mu $ is any Radon measure. Suppose that  $f$ satisfies 
\begin{equation}
\frac{1}{|Q|}\int_{Q} |f-f_{Q}| \le a(Q)
\end{equation}
for every cube $Q$. Then, there exists a dimensional constant $C$ such that for any cube $Q$
\begin{equation*}
\left( \frac{1}{ w(Q)  } \int_{ Q }   |f -f_{Q}|^{p_w^*}     \,wdx\right)^{\frac{1}{p_w^*}}  \, \leq C a(Q).
\end{equation*}
\end{theorem}

As an immediate corollary we obtain a weighted Poincar\'e-Sobolev $(p_w^*,p)$ inequality. 

\begin{corollary}\label{cor:ptimes-Aq}
Let $1 \le p < n $ and let $w\in A_q$ with $1\le q\le p$. Let $p_w^*$ as in the previous theorem. Then the following inequality holds.
\begin{equation*}\label{eq:ptimes-Aq-gradient}
\left (\frac{1}{w(Q)}\int_Q |f-f_Q|^{p_w^*}wdx\right )^{\frac{1}{p_w^*}} \leq [w]^{\frac{1}{p}}_{A_p}\ell(Q)\left (\frac{1}{w(Q)}\int_Q|\nabla f|^{p} w dx\right )^{1/p}.
\end{equation*}
\end{corollary}

Let us discuss briefly the different results depending on which class of weights is $w$. The constant on the inequality is always $[w]^{\frac{1}{p}}_{A_p}$ and this will not improve by assuming $q<p$ in the corollary above. The difference will appear on the value of $p_w^*$. For example, when $q=1$ we obtain an $A_1$ inequality of the form 
\begin{equation*}
\left (\frac{1}{w(Q)}\int_Q |f-f_Q|^{p_w^*}\, w\ dx\right )^{\frac{1}{p_w^*}} \leq [w]^{\frac{1}{p}}_{A_p}\ell(Q)\left (\frac{1}{w(Q)}\int_Q|\nabla f|^{p} \, w \ dx\right )^{1/p}
\end{equation*}
with 
$$
\frac{1}{p} -\frac{1}{ p_w^* }=\frac{1}{n(1+\log [w]_{A_1})}.
$$

This result should be compared to Corollary \ref{thm:P(p,p*)-local-avg-A1}, where we obtain by means of other arguments,  that $p^*=\frac{pn}{n-p}$ (which is equivalent to $\frac{1}{p}-\frac{1}{p^*}=\frac{1}{n}$)  instead of $p_w^*$ and linear dependence on the $A_1$ constant $[w]_{A_1}$ on the right hand side of the inequality.
Here we improve the constant by replacing the linear $[w]_{A_1}$ constant by $[w]^{\frac{1}{p}}_{A_p}$ but, on the other hand, we do not reach the usual Sobolev exponent $p^*$ associated to the constant weight. We will further improve on this in Corollary \ref{cor:MainCoro}.

\subsection{Generalized    Poincar\'e-Sobolev and the good-\texorpdfstring{$\lambda$}{lambda} method} 

The Sobolev exponent obtained in Theorem \ref{cor:ptimes-Aq} can be improved, namely, we can obtain  larger values of $p_w^*$ for the  $(p_w^*,p)$ Poincar\'e type inequalities. However, we have to pay with some extra powers of the $A_q$ constants in front. The reason behind that is that we will be using the very well known method of ``good-$\lambda$'' inequalities of Burkholder and Gundy  in a similar way as done in \cite{FLW, MacManus-Perez-98}.

We have the following theorem. Since we will be using both $D_r(w)$ and $SD_p^n(w)$ conditions on the functional $a$, we emphasize the difference between them by using the notation $\|a\|_{D_r(w)}$ for the first.
\begin{theorem}\label{thm:goodL-aQ}
Let $a$ be a functional satisfying: 
\begin{enumerate}
\item For some $p\ge 1$ it satisfies condition $SD^n_p(w)$ from \eqref{eq:SDp} with norm 1. Namely, for any cube $Q$ and any family $\{Q_i\}$ of pairwise disjoint subcubes of $Q$ such that $\{Q_i\}\in S(L)$ with $L>1$, the following inequality holds:
\begin{equation*}
\sum_{i}a(Q_i)^pw(Q_i)\le  \left (\frac{1}{L}\right )^{\frac{p}{n} }a(Q)^pw(Q).
\end{equation*}
\item For some $r>p$ the functional $a$ satisfies the $D_{r}(w)$ condition \eqref{eq:Dp}, that is, there is a constant $\|a\|_{D_r(w)}$ such that for any cube $Q$ and any family $\Lambda$ of pairwise disjoint subcubes of $Q$, the following inequality holds:
\begin{equation*}
\sum_{P\in\Lambda}a(P)^{r}w(P)\leq \|a\|^{r}_{D_r(w)} a(Q)^{r}w(Q).
\end{equation*}
\end{enumerate}

Let $f$ be a locally integrable function such that
\begin{equation*}
\frac{1}{|Q|}\int_{Q} |f-f_{Q}| \le a(Q)
\end{equation*}
for every cube $Q$. 

Let $w\in A_p$. Then, there exists a constant $c$ such that for any cube $Q$
\begin{equation*}
\| f- f_{Q} \|_{ _{L^{r,\infty}(Q, \frac{wdx}{w(Q)}  )} } \leq c\, \|a\|_{D_r(w)} [w]^{\frac1p}_{A_p}\,a(Q).
\end{equation*}
\end{theorem}

Now, combining this result with the weak-implies-strong argument, we obtain the following corollaries on Poincar\'e Sobolev inequalities.

\begin{corollary}\label{cor:MainCoro}

Let $1 \le p < n $. Let $w\in A_q$ with $1\le q\le p$. Define the exponent $p_w^*$ by the formula

$$
\frac{1}{p} -\frac{1}{p_w^* }=\frac{1}{nq}.
$$
Then the following Poincar\'e-Sobolev $(p_w^*,p)$,inequality holds
$$
\left (\frac{1}{w(Q)}\int_Q |f-f_Q|^{p_w^*} w dx\right )^{\frac{1}{p_w^*}} \leq C [w]^{ \frac{1}{nq}}_{A_q}  
[w]_{A_p}^{ \frac{2}{p}} 
 \ell(Q)\left (\frac{1}{w(Q)}\int_Q|\nabla f|^{p}  w  dx\right )^\frac{1}{p},
$$
where $C=C_{n,p}$.
\end{corollary}

\begin{remark}
We remark that this corollary improves the main result for $A_p$ weights from section 15.26 of \cite{HKM}. 
\end{remark}

\begin{remark}
We also remark here that in the case of $w\in A_1$, namely under a stronger condition, we recover the classical Sobolev index $p_w^*=p^*=\frac{np}{n-p}$,
$$
\left (\frac{1}{w(Q)}\int_Q |f-f_Q|^{p^*} w dx\right )^{\frac{1}{p^*}} \leq C [w]^{ \frac{1}{n}}_{A_1}  
[w]_{A_p}^{ \frac{2}{p}} 
 \ell(Q)\left (\frac{1}{w(Q)}\int_Q|\nabla f|^{p}  w  dx\right )^\frac{1}{p},
$$
with $C=C_{n,p}$. That is, the $A_1$ class of weights  behaves in that aspect as the Lebesgue measure. 
This improves the result given in Lemma 15.30 p. 308 of \cite{HKM} since we are able  to reach both optimal endpoints, namely the exponent $p=1$  and the (unweighted) Sobolev exponent $p^{*}=\frac{pn}{n-p}$, for any $1 \le p <n$.  This result should be compared with Corollary \ref{thm:P(p,p*)-local-avg-A1} below where we also derive by different methods a similar estimate with a worst constant in the range 
$2n'<p<n$. 

\end{remark}

The proof of  Corollary  \ref{cor:MainCoro}  essentially reduces to check that the generic functional given in \eqref{eq:general-a(Q)} (which includes the case of the gradient) satisfies the hypothesis $(1)$ and $(2)$ from Theorem \ref{thm:goodL-aQ} involving an explicit estimate of $\|a\|$. The fact that the functional 
\begin{equation*}
a(Q)=\ell(Q)\left(\frac{1}{w(Q)}\mu(Q) \right)^{1/p}
\end{equation*}
satisfies condition (1) will be obtained in Lemma \ref{lem:L-small}. An appropriate value of $r$ satisfying condition (2)  and the norm estimate for the functional $a$ is the content of Lemma \ref{lem:new-smallAp} below.

\begin{lemma} \label{lem:new-smallAp} 

Let $1\leq  p<n$    and let $a(Q)$ defined as in \eqref{eq:general-a(Q)} with $\alpha=1$, namely 
$$a(Q)=\ell(Q)\left(\frac{1}{w(Q)}\mu(Q) \right)^{1/p}.
$$
Let  $w \in A_q$ with $1\leq q \leq p$.  Define the exponent $p^*_w$ by the formula
$$
\frac{1}{p} -\frac{1}{p_w^*  }=\frac{1}{nq}.
$$
Then for any family $\{Q_i\}$ of pairwise disjoint subcubes of $Q$ the following inequality holds:
\begin{equation}
\sum_{i}a(Q_i)^{ p_w^*   }\,w(Q_i) \leq  [w]_{A_q}^{ \frac{ p_w^*   }{nq} }
\,a(Q)^{ p_w^* } w(Q).
\end{equation}
That is, the functional $a$ satisfies the condition $D_{p_w^*}(w)$ and further we have that $\|a\| \leq [w]_{A_q}^{ \frac{ 1 }{nq} }$.

\end{lemma}

We also will show that our method provides similar results in the scale of Lorentz spaces. More precisely, we obtain similar inequalities as in Corollary \ref{cor:Poincare(p,p)-twoweight}  assuming that the gradient satisfies a stronger condition in the (local) Lorentz space $L^{p,1}$ 
but with a larger class of weights $A_{p,1}$  which contains $A_p$. This class of weights was introduced by Chung-Hunt-Kurtz in \cite{CHK} (see the precise definitions in Section \ref{sec:Lorentz}).

\begin{corollary}\label{cor:Lp1(w)-a(Q)-clean}

Let  $w\in A_{p,1}$, then there exists a dimensional constant $c_n$ such that for any cube $Q$
\begin{equation*}
\left( \frac{1}{ w(Q)  } \int_{ Q }   |f -f_{Q}|^p     \,wdx\right)^{\frac{1}p}   \leq  c_n  \ell(Q) [w]_{A_{p,1}}^{\frac1p} \, \,\Big\| \nabla f \Big\|_{L^{p, 1}(Q, \frac{wdx}{w(Q)} ) }.
\end{equation*}
\end{corollary}

\begin{remark}

As in the $L^p$ case it would be possible to derive some  Sobolev-Poincar\'e inequalities like in Theorem \ref{thm:ptimes-Aq} or \ref{thm:goodL-aQ} but we will not pursue in this direction. 

\end{remark}

By means of a completely different method, we will present  a special  two weight Poincar\'e-Sobolev inequality  involving the maximal function on the right hand side where no assumption is assumed on the weight. As usual, we will denote  by $p^*$  be the classical Sobolev exponent, 
$$
\frac{1}{p} -\frac{1}{p^*}=\frac{1}{n}   \qquad 1 \leq p< n.
$$
We have the following theorem.

\begin{theorem}\label{thm:PoincareSobolev-2weights}
Let $w$ be a weight in $\mathbb{R}^n$, $n\ge 2$. Then if $1\leq p <n$ we have that
\begin{equation}\label{eq:P(p,p*)-local-avg}
\left (\int_Q|f-f_{Q,w}|^{p^*} wdx \right )^{\frac{1}{p*}} \le C \left (\int_Q |\nabla f|^p \frac{(M^c(w\chi_Q))^\frac{p}{n'}}{w^{p-1}} dx \right )^{\frac{1}{p}}.
\end{equation} 
\end{theorem}

From that basic estimate, we derive a quantitative Poincar\'e-Sobolev inequality for $A_1$ weights. Further, we explore the sharpness of the result.

\begin{corollary}\label{thm:P(p,p*)-local-avg-A1}
Under the same hypothesis of Theorem \ref{thm:PoincareSobolev-2weights}, if in addition we have that $w\in A_1$, then
\begin{equation}\label{eq:P(p,p*)-local-avg-A1}
\left (\frac{1}{w(Q)}\int_Q|f-f_{Q,w}|^{p*}w\right )^\frac{1}{p*}\leq c\, [w]_{A_1}\ell(Q)\left (\frac{1}{w(Q)}\int_Q |\nabla f|^p w\right )^\frac{1}{p}.
\end{equation} 

 Furthermore, the result is sharp in the case $p=1$ in the sense that $[w]_{A_1}$ cannot be replaced by $\psi([w]_{A_1})$ with $\psi:[1,\infty)\to (0,\infty)$  satisfying $\psi(t)=o(t)$ as $t\rightarrow\infty$. 

\end{corollary}

 As mentioned before,  the conclusion of Corollary  \ref{thm:P(p,p*)-local-avg-A1} should be compared to the result in Corollary \ref{cor:MainCoro}, where this result is improved for large enough values of  $n$ and $p$ (more precisely $p\in (2n',n)$, $n>3$). It should be also compared to the necessary condition given in Proposition \ref{pro:LowerBoundBeta} regarding  the best possible exponent on the $A_1$ constant  in the case $p>1$. We will present an example showing that the best possible exponent for the $A_1$ constant in an inequality like \eqref{eq:P(p,p*)-local-avg-A1} is $1/p$. That  suggests that the conjecture on weighted Poincar\'e inequalities for $A_1$ weights should be the following:

\begin{conjecture}\label{conj:A1conjecture}
Let $w$ be an $A_1$ weight in $\mathbb{R}^n$, $n\ge 2$. Then if $1\leq p <n$ 
$$
\left (\frac{1}{w(Q)}\int_Q|f-f_{Q,w}|^{p*}w\right )^\frac{1}{p*}\le C [w]^{\frac{1}{p}}_{A_1}\ell(Q)\left (\frac{1}{w(Q)}\int_Q |\nabla f|^p w\right )^\frac{1}{p}.
$$
\end{conjecture}

\subsection{ Generalized  Poincar\'e inequalities with polynomials}

As we already mentioned we  also present an extension of our main result in Theorem \ref{thm:Lp(w)-a(Q)-clean} that can be used to obtain self-improving properties within the context of generalized  Poincar\'e inequalities with polynomials. This is intimately related to understanding higher order derivative estimates where  the \emph{truncation method} is not available.

 Let $m\in \mathbb{N}$. We denote by $P_Q f$ the projection of the function $f$ over the space $\mathcal{P}_m$ of polynomials of degree at most $m$ in $n$ variables on $Q$ (see Section \ref{sec:higher} for a more detailed discussion).

\begin{theorem}\label{thm:higher-a(Q)}
Let $w$ be any $A_\infty$ weight. Consider also the functional   $a$ such that for some $p \ge 1$ it satisfies condition $SD^s_p(w)$ from \eqref{eq:SDp} with $s>1$ and constant $\|a\|$. Let $f$ be a locally integrable function such that
\begin{equation*}
\frac{1}{|Q|}\int_{Q} |f-P_{Q}f| \le a(Q),
\end{equation*}
for every cube $Q$. Then, there exists a dimensional constant $C_{n,m}$ such that for any cube $Q$
\begin{equation*}
\left( \frac{1}{ w(Q)  } \int_{ Q }   |f -P_{Q}f|^p     \,wdx\right)^{\frac{1}p}   \leq  C_{n,m} 2^\frac{s+1}{p'}s \|a\|^s  a(Q) 
\end{equation*}
where $C_n$ is a dimensional constant. When $p=1$, the factor $2^\frac{s+1}{p'}$ vanishes.
\end{theorem}

As a consequence of this theorem, we are now able to present a proof of inequality \eqref{eq:Poincare(pp)-higher} which is an extension of the result in Corollary \ref{cor:Poincare(p,p)-twoweight}.                                                          

\begin{corollary}\label{cor:Poincare(pp)-higher}
Let  \,$(u,v)\in A_p,  u\in A_\infty$.   Then the following Poincar\'e $(p,p)$ inequality holds
$$
\left (\frac{1}{u(Q)}\int_Q |f- P_Qf |^p\, u\ dx\right )^{\frac{1}{p}}\le C[u,v]^{\frac{1}{p}}_{A_p}\ell(Q)^m \left (\frac{1}{u(Q)}\int_Q |\nabla^m f|^{p} \, v \ dx\right )^{\frac{1}{p}},
$$
where $C=C_{n,m}$ is a dimensional constant.
\end{corollary}

It is also clear that we can derive some results in the spirit of Theorem \ref{thm:ptimes-Aq} but we will not pursue in this direction.

\subsection{  Lack of Poincar\'e inequalities for \emph{all} \texorpdfstring{$RH_{\infty}$}{Ainfty} weights and the \emph{failure} of Poincar\'e inequalities in the range  \texorpdfstring{$p<1$}{p<1}  }

In this paper we also address the problem of characterizing the class of weight functions such that a Poincar\'e $(p,p)$ inequality holds. Far from being able to provide a complete answer, we include the following negative result which says that the class $RH_\infty$ is too big.  We recall that a weight $w$ belongs to the the class $RH_\infty$ if there is a constant $c$ such that
$$
\sup_Q w \leq c\ \avgint_{Q} w. 
$$
This definition means that $w$ satisfies a reverse H\"older's inequality for any exponent and hence $RH_\infty \subset A_\infty$. 
It is well known that $|\pi| \in RH_\infty$ when $\pi$ is a polynomial in $\mathbb{R}^n$.  It is also known that $(M\mu)^{-\lambda} \in RH_{\infty}$ 
(see Lemma \ref{RHinfity})

\begin{theorem}\label{thm:RHinfty} Let $1\le p <\infty$ and suppose that a Poincar\'e $(p,p)$ inequality holds for the class of weights 
$RH_{\infty}$, namely that 

\begin{equation}\label{eq:P(p,p)-RHinfty}
\inf_a \left (\int_Q|f-a|^p \,wdx\right )^\frac{1}{p}\le c\,\ell(Q)\left (\int_Q |\nabla f|^p \,w dx\right )^\frac{1}{p} \qquad w \in RH_{\infty}
\end{equation}
with constant $c$ independent of the cube $Q\subset \mathbb{R}^n$.  Then, for every $0<q<1$ it also holds that 
\begin{equation}\label{eq:P(p,p)-p<1}
\inf_a \left (\int_Q|f-a|^q \ dx\right )^\frac{1}{q}\le c\,\ell(Q)\left (\int_Q |\nabla f|^q \ dx\right )^\frac{1}{q}, 
\end{equation}
with constant $c$ independent of $Q$.  Since this is {\bf false},  \eqref{eq:P(p,p)-RHinfty} cannot hold for \emph{every} $w \in RH_{\infty}$. In particular, it follows that
\eqref{eq:P(p,p)-RHinfty} cannot hold for \emph{every} $w \in A_{\infty}$
\end{theorem}

We recall here that there is an example from \cite[p.224]{BK94} proving that \eqref{eq:P(p,p)-p<1} fails in general. 
The proof of Theorem \ref{thm:RHinfty}, which can be seen as an application of some extrapolation type arguments, is presented in Section \ref{sec:noAinfinity}.

\subsection{Outline}\label{sec:outline}
The organization of the paper is as follows. Section \ref{sec:prelim} is devoted to collect some notation and well known results as well as to summarize some previous and auxiliary results.
In Section \ref{sec:Ap} we develop the general quantitative theory of self-improving functionals $a(Q)$ discussed above, proving Theorem \ref{thm:Lp(w)-a(Q)-clean} and it consequences. In Section \ref{sec:KZ} we discuss the results related to Keith-Zhong's theorem. We show what kind of higher exponents can be reached using Theorem \ref{thm:Lp(w)-a(Q)-clean} in Section \ref{sec:PoincareSobolev}. The approach involving the good-$\lambda$ technique is contained in Section \ref{sec:goodL}, where we prove Theorem \ref{thm:ptimes-Aq} and obtain Corollary \ref{cor:MainCoro}. In Section \ref{sec:mixed} we prove the mixed weight Poincar\'e inequality from Theorem \ref{thm:PoincareSobolev-2weights} and derive from it the $A_1$ result providing the proof of Corollary \ref{thm:P(p,p*)-local-avg-A1}. We also discuss the best possible dependence on $[w]_{A_1}$. In Section \ref{sec:higher} we discuss some extensions and applications of our methods to two-weight Poincar\'e inequalities and to inequalities involving higher derivatives. In Section \ref{sec:Lorentz} we present applications to Poincar\'e inequalities on Lorentz spaces. Section \ref{sec:noAinfinity} contains the negative result from Theorem \ref{thm:RHinfty}. 

We also include an appendix in Section \ref{sec:appendix} for completeness. There we present the 
connections between Poincar\'e inequalities and fractional integrals and we also include the so called \emph{truncation method} which yields  that an appropriate weak type estimate implies the corresponding strong inequality.

\section{Preliminaries and some well known results}\label{sec:prelim}

Recall that we are interested in proving \emph{weighted} Poincar\'e inequalities of the form
\begin{equation}\label{eq:TemplatePoincare}
\left (\frac{1}{w(Q)}\int_Q|f-f_Q|^{q}w\right )^\frac{1}{q}\le C(w)\ell(Q)\left (\frac{1}{w(Q)}\int_Q |\nabla f|^p w\right )^\frac{1}{p},
\end{equation}
where  $1\le q,p\le \infty$ and $w$ is a weight function, i.e., a locally integrable non-negative function.  As usual, we will denote by $\avgint_E f \ dx=f_E=\frac{1}{|E|}\int_E f \ dx$ the average of $f$ over $E$ with respect to the Lebesgue measure. For a given measure $\mu$ defined for every cube $Q$, we will denote $f_{Q,\mu} = \avgint_Q f d\mu:= \frac{1}{\mu(Q)}\int_Q fd\mu$. In the particular case of densities given by a weight $w$, we will write $f_{Q,w}=\frac{1}{w(Q)}\int_Q fwdx$.

A brief remark on the oscillation on the left hand side is in order. One should try to prove the inequality for the oscillation with respect to the weighted average $f_{Q,w}$, but we can, essentially, prove inequalities like \eqref{eq:TemplatePoincare} with any constant $c$ instead of the average $f_Q$. This is a consequence of the fact that 
$$
\left (\frac{1}{w(Q)}\int_Q|f-f_{Q,w}|^{q}w\right )^\frac{1}{q}\le 2\inf_{c\in \mathbb{R}}\left (\frac{1}{w(Q)}\int_Q|f-c|^{q}w\right )^\frac{1}{q}.
$$

 We are particularly interested in the quantitative analysis on the constant $C(w)$ when the function $w$ belongs to a certain class of weights. As already mentioned, after the work \cite{FKS},  a natural scenario for this analysis is the class of Muckenhoupt $A_p$ weights defined as follows. 
For $1<p<\infty$ and $p'=p/(p-1)$, a weight $w$ is in $A_p$, $1<p<\infty$, if 
\begin{equation}\label{eq:Ap}
[w]_{A_p}:=\sup_Q \left(\avgint_Q w\,dx\right) \left(\avgint_Q w^{1-p'}\,dx\right)^{p-1} <\infty,
\end{equation}
where the supremum is taken over all cubes $Q\subset \mathbb{R}^n$ with sides parallel to the coordinate axes.
The limiting case of \eqref{eq:Ap} when $p=1$, defines the class $A_1$; that is, the set of weights $w$ such that
\begin{equation*}
[w]_{ A_1}:=\sup_{Q}\bigg(\avgint_Q w\,dx \bigg) \esssup_{Q} (w^{-1})<+\infty.
\end{equation*}
This is equivalent to $w$ having the property
\begin{equation*}
 Mw(x)\le [w]_{A_1}w(x)\qquad \text{ a.e. } x \in \mathbb{R}^n,
\end{equation*}
where $M$ denotes the maximal function:
\begin{equation*}\label{eq:M}
Mf(x) = \sup_{Q\ni x } \avgint_{Q} |f(y)|\ dy, 
\end{equation*}
and the supremum is taken over all cubes $Q\subset \mathbb{R}^n$ with sides parallel to the coordinate axes containing the point $x$. The \emph{centered} version with respect to euclidean balls of this operator is defined as
\begin{equation*}\label{eq:Mc}
M^cf(x) = \sup_{r>0} \avgint_{B(x,r)} |f(y)|\ dy.
\end{equation*}
Since we are considering the euclidean space $\mathbb{R}^n$ endowed with the Lebesgue measure, both maximal operators defined above are pointwise comparable, up to a dimensional constant.

The other limiting case of \eqref{eq:Ap} corresponds to the case $p=\infty$. Although it is very convenient to define this class, also introduced by Muckenhoupt,  by 
$$
A_\infty:=\bigcup_{p\ge 1} A_p,
$$
it is also  characterized by means of this constant 
\begin{equation*}
[w]_{A_\infty}:=\sup_Q\frac{1}{w(Q)}\int_Q M(w\chi_Q )\ dx,
\end{equation*}
(see \cite{HPR1} and \cite{HP}).

An important feature of $A_\infty$ weights is the so called reverse H\"older inequality (RHI) which can be seen as a selfimproving property of the local integrability of the weight $w$. More precisely, the following Theorem was proved in \cite{HP}  (see also \cite{HPR1}  for a simpler proof).

\begin{theorem}\label{thm:RHI}
Let $w\in A_\infty$. Define $r_w=1+\frac{1}{2^{n+1}[w]_{A_\infty}-1}$. Then for any cube $Q$, we have that
\begin{equation}\label{eq:RHI}
\avgint_Q w^{r_w} \le 2 \left (\avgint_Q w\right )^{r_w}.
\end{equation}
\end{theorem}

We will need some additional well known properties of $A_p$ weights that we include here. Using H\"older's inequality with $p$ and its conjugate $p'$, we have that for every cube $Q$ and every $g\geq 0$,
\begin{equation}\label{eq:prop-Ap-function}
\frac{1}{|Q|}\int_Q g\ dx\leq[w]^{\frac{1}{p}}_{A_p} \left( \frac{1}{w(Q)}\int_Q g^pw\ dx\right )^{\frac{1}{p}}.
\end{equation}

Specializing inequality \eqref{eq:prop-Ap-function} for $g\equiv\chi_{E}$ we obtain that, for any measurable set $E\subset Q$, that
\begin{equation}\label{eq:prop-Ap-set}
\frac{|E|}{|Q|}\leq[w]^{\frac{1}{p}}_{A_p}\left (\frac{w(E)}{w(Q)}\right )^{\frac{1}{p}}.
\end{equation}

\begin{remark}
We remark here the interesting fact that the condition above characterizes the so called 
$A_{p,1}$ condition, which is strictly weaker than $A_p$. For example, $(M\mu)^{1-p}$ belongs to $A_{p,1}$ but not to $A_p$. Moreover, the $[w]_{ A_{p,1} }$ constant, defined as the best possible constant in the inequality 
\begin{equation*}
\left (\frac{|E|}{|Q|}  \right )^{p}\leq C \frac{w(E)}{w(Q)},
\end{equation*}
could be strictly smaller than $[w]_{ A_{p} }$. This smaller class seems to be the right class to study weighted Poincar\'e inequalities in the context of Lorentz spaces (see Section \ref{sec:Lorentz}). 
\end{remark}

We also will use few times the class of pair of $A_p$ weights. 
\begin{definition}
Given a pair of weights $(u,v)$, we will say that it belongs to the $A_p$ class and denote $(u,v)\in A_p$ if
\begin{equation}\label{eq:Ap-two-weights}
[u,v]_{A_p}:=\sup_Q \left(\avgint_Q u\,dx\right) \left(\avgint_Q v^{1-p'}\,dx\right)^{p-1} <\infty.
\end{equation}
\end{definition}

The following well known result will be used 
\begin{equation}\label{charactAp-two-weights}
\|M\|_{L^p(v)\to L^{p,\infty}(u)} \approx  [u,v]^{1/p}_{A_p}, 
\end{equation}
with a dimensional constant in front.

It is also well know that 
\begin{equation}\label{two weight Ap charact}
\frac{1}{|Q|}  \int_Q f\ dx \leq [u,v]_{A_p}^{ \frac{1}{p} }  
\left( \frac{1}{u(Q)}\int_Q f^p v\ dx\right )^{\frac{1}{p}}   \qquad f \geq 0.
\end{equation}

We finish this section by setting some notation related to certain function spaces and normalized local norms. We will consider first some basics about Lorentz spaces.  Let $\mu$ be a Radon measure on $\mathbb{R}^n$. A function $f$ belongs to the Lorentz space $L^{p,q}(\mu)$, $0 < p,q \le \infty$, if
\[ \norm{f}_{ L^{p,q}(\mu) } = \left[ p\, \int_{0}^{\infty}  \left(
t\, \mu\{  x \in \R^{n}: \absval{f(x)} > t \}^{1/p} \right)^{q}\,
\frac{dt}{t} \right]^{1/q} < \infty,
\]
whenever $q < \infty$, and
\[ \sup_{ 0
< t < \infty}{t\, \mu\{  x \in \R^{n}: \absval{f(x)} > t \}^{1/p}} < \infty,
\]
if $q = \infty$. Since we will be dealing with the measure $\mu=w dx$ given by a weight, we will denote
$$
\|f\|_{L^{p,\infty}_{w}} \leq \|f\|_{L^{p,p}_{w}}\leq c\,\|f\|_{L^{p,1}_{w}}.
$$
We also have Holder's inequality in this setting (see \cite{KS}):
\begin{equation}
\int_{\R^{n}}  \absval{ f(y)  g(y) }\, d\mu(y) \leq \, 4
\norm{f}_{ L^{p,q}(\mu) } \, \norm{f}_{L^{p',q'}(\mu)}.
\label{GHI}
\end{equation}

We will consider the case $q=1$  and $q=\infty$ and the following notation for local averaging will be used
$$
\Big\| g \Big\|_{L^{p, 1}(Q, \frac{wdx}{w(Q)} ) }= \int   \limits_0^{\infty}   \left( \frac{1}{w(Q)}  w\{x\in Q:  g(x)>t\}\right) ^{1/p} \,dt.
$$
and similarly for  $\Big\| g \Big\|_{L^{p, q}(Q, \frac{wdx}{|Q|} ) }$.

There is a class of weights attached to some of these spaces denoted by $A_{p,1}$  that was introduced by Chung-Hunt-Kurtz in \cite{CHK} and that will be used. We will denote by $A_{p,1}$, $p>1$, to the class of weights $w$ for which the quantity
\begin{equation}
[w]_{A_{p,1}}:= \sup_{Q} \left(  \avgint_Q  w(y)\ dy \right)  \Big\| \frac{1}{w}\Big\|_{L^{p',\infty}(Q, \frac{wdx}{|Q|} ) }^{p}.
\label{Ap,1}
\end{equation}
is finite. See also \cite[p. 145]{KK}.  Observe that 
$$
[w]_{A_{p,1}} \leq [w]_{A_p}. 
$$
and then $A_p \subset A_{p,1}$ with strict  inclusion since  if $M\mu$ is a.e. finite and $p>1$,  it is known that $(M\mu)^{1-p} \in A_{p,1}$ with
$$
[(M\mu)^{1-p}]_{A_{p,1}} \leq c_n,
$$
although in general
$$
[(M\mu)^{1-p}]_{A_{p}} =\infty.
$$
It is also known  that 
$ A_{p,1} \subset  \bigcap_{q>p} A_q$.

\section{The proof of Theorem \ref{thm:Lp(w)-a(Q)-clean} and its corollaries}\label{sec:Ap}

Now we are able to present the proof of the main result of this section, Theorem   \ref{thm:Lp(w)-a(Q)-clean}, providing a general self-improving result related to a generic functional $a$.


\begin{proof}[Proof of Theorem \ref{thm:Lp(w)-a(Q)-clean}] 
We recall that the goal is to prove the inequality
\begin{equation}\label{eq:FirstMainEstimate-Proof}
\left( \frac{1}{ w(Q)  } \int_{ Q }   |f -f_{Q}|^p     \,wdx\right)^{\frac{1}p}  \, \leq  c_n s \|a\|^s a(Q),
\end{equation}
provided that 
\begin{equation}
\frac{1}{|Q|}\int_{Q} |f-f_{Q}| \le a(Q),\label{eq:UnWeightedStartingPointL1}
\end{equation}
and that the functional   $a$ is such that for some $p\ge 1$ it satisfies condition $SD^s_p(w)$ from \eqref{eq:SDp} with $s>1$ and constant $\|a\|$.

We will assume that $f$ is bounded.  As  a first step to prove  \eqref{eq:FirstMainEstimate-Proof} we claim  the following a priori estimate 
\begin{equation} \label{eq:a priori estimate}
\sup_Q \left( \frac{1}{ w(Q)  } \int_{ Q }   \frac{|f -f_{Q}|^p}{a(Q)^p}   \,wdx\right)^{\frac{1}{p} }  <\infty. 
\end{equation}
To do this we consider first an approximation 
$$
X_{\varepsilon}=\sup_Q \left (\frac{1}{w(Q)}\int_{Q}\left |\frac{f-f_{Q}}{a_{\varepsilon}(Q)}\right |^p \, w dx \right )^{1/p}.
$$
where $a_{\varepsilon}(Q)=a(Q)+\varepsilon$.

We need the following lemma. Recall that the hypothesis on the functional $a$ is that it satisfies a smallness preservation condition. A difficulty is that  the the perturbed functional $a_\varepsilon$ has a worst smallness exponent.

\begin{lemma}\label{lem:SDps-a(Q)+eps}
Let $w$ be a $A_{\infty}$  and let $a \in SD^s_p(w)$ $s>1$. Suppose that  $\{Q_i\}\in S(L)$, then there are  constants $C$ and $S$ larger than $\|a\|$ and $s$ respectively  such that
\begin{equation*} 
\sum_{i}a_{\varepsilon}(Q_i)^pw(Q_i)\leq C^p \left (\frac{1}{L}\right )^{\frac{p}{S}}a_{\varepsilon}(Q)^pw(Q).
\end{equation*}
\end{lemma}

\begin{proof}
We compute the sum from the $SD^s_p(w)$ condition:
\begin{eqnarray*}
 \left( \sum_{i}a_{\varepsilon}(Q_i)^pw(Q_i)  \right)^{\frac1p } & = & 
\left( \sum_{i} (a(Q_i) +\varepsilon )^pw(Q_i) \right)^{\frac1p }\\
&\le & \left( \sum_{i} a(Q_i)^pw(Q_i) \right)^{\frac1p } + \left( \sum_{i} \varepsilon^pw(Q_i) \right)^{\frac1p }\\
&\le & \frac{\|a\|}{L^{1/s}}  a(Q)w(Q)^{\frac1p } + \varepsilon w\left( E_Q \right)^{\frac1p }
\end{eqnarray*}
where $ E_Q =:\bigcup_{i}Q_i \subset Q$. Recall that by the smallness condition we have that
$$
|E_Q|\le \frac{|Q|}{L}.
$$
By Holder's inequality and the RHI for $w$ with the exponent $r$ from Theorem \ref{thm:RHI} (since $w\in A_{\infty}$), 
\begin{equation*}
w(E_Q)\leq 2\,\left(\frac{|E_Q|}{|Q|} \right)^{\frac{1}{r'_w}} \, w(Q).
\end{equation*}
Now, since $r' \approx [w]_{A_{\infty}}$ we obtain 
$$
w(E_Q)\leq 2\,\left(\frac1L \right)^{\frac{1}{r'_w}} \, w(Q).
$$
Putting things together

\begin{eqnarray*}
 \left( \sum_{i}a_{\varepsilon}(Q_i)^pw(Q_i)  \right)^{\frac1p } 
&  \leq & \frac{\|a\|}{L^{1/s}}  a(Q)w(Q)^{\frac1p } + \varepsilon  \frac{2^{\frac1p }}{L^{1/pr'}} \, w(Q)^{\frac1p }\\
&  \leq & \max\{ \frac{\|a\|}{L^{1/s}}, \frac{2^{\frac1p }}{L^{1/pr'}}  \} \,a_{\varepsilon}(Q)w(Q)^{\frac1p } \\
& \leq &    \frac{\max\{ \|a\|, 2^{\frac1p } \}}{   L^{1/\max\{s,pr'\}}      } \,a_{\varepsilon}(Q)w(Q)^{\frac1p }.
 \end{eqnarray*}
The proof of the lemma is now complete.
\end{proof}

We consider the local Calder\'on-Zygmund decomposition  of $\frac{|f-f_{Q}|}{a_\varepsilon(Q)}$ relative to $Q$ at level $L$ on $Q$ for a large universal constant $L>1$ to be chosen. Let $\mathcal{D}(Q)$ be the family of dyadic subcubes of $Q$. The Calder\'on-Zygmund (C-Z) decomposition yields a collection $\{Q_{j}\}$ of cubes such that $Q_j\in \mathcal{D}(Q)$, maximal with respect to inclusion, satisfying

\begin{equation}\label{eq:CZ1}
L  < \frac{1}{|Q_{j}|  }\int_{Q_{j}} \frac{|f-f_{Q}|}{{a_\varepsilon(Q)}} \, dy. 
\end{equation}
Then, if $P$ is dyadic with $P \supset Q_j$ 
\begin{equation}\label{eq:CZ2}
\frac{ 1 }{|P|  }
\int_{P} \frac{|f-f_{Q}|}{{a_\varepsilon(Q)}} \, dy   \leq L
\end{equation}
and hence 
\begin{equation}\label{eq:CZ3}
L  < \frac{ 1 }{|Q_{j}|  }
\int_{Q_{j}} \frac{ |f-f_{Q}| }{{a_\varepsilon(Q)}} \, dy \leq L\,2^{n}
\end{equation}
for each integer $j$.  Also note that 
$$  \left \{x\in Q: M_Q^d\left ( \frac{|f-f_{Q}|}{{a_\varepsilon(Q)}}\chi_{  Q }\right )(x) > L   \right \} = \bigcup_{j}Q_j=:\Omega_L
$$
where $M^d_Q$ stands for the dyadic maximal function adapted to the cube $Q$. That is, 
$$
M^d_Q(f)(x):=\sup_{P\ni x}\avgint_P |f(y)|\ dy\qquad x\in Q, P\in \mathcal{D}(Q).
$$
Then, by the Lebesgue differentiation theorem it follows that
$$\frac{|f(x)-f_{Q}|}{{a_\varepsilon(Q)}} \leq L   \qquad              a.e. \ x \notin \Omega_L
$$

Also, observe that by \eqref{eq:CZ1} (or the weak type $(1,1)$ property of $M$) and recalling our starting assumption \eqref{eq:UnWeightedStartingPointL1}, we have that $\{Q_i\}\in S(L)$, namely
\begin{equation*}\label{eq:CZ4a}
|\Omega_L|=|\bigcup_{j}Q_j | < \frac{|Q|}{L}.
\end{equation*}

Now, given the C-Z decomposition of the cube $Q$, we perform the classical C-Z of the function $\dfrac{f-f_Q}{{a_\varepsilon(Q)}}$ as
\begin{equation}\label{eq:CZ-f-fQ}
\frac{f-f_Q}{a_\varepsilon(Q)}=g_Q+b_Q,
\end{equation}
where the functions $g_Q$ and $b_Q$ are defined as usual. We have that

\begin{equation}\label{eq:gQ}
g_Q(x) = \left \{
\begin{array}{ccc}
\dfrac{f-f_Q}{{a_\varepsilon(Q)}} & , &  x \notin \Omega_L \\
&&\\
\displaystyle \avgint_{Q_i}\dfrac{f-f_Q}{{a_\varepsilon(Q)}} & , &  x\in \Omega_L,  x\in Q_i
\end{array}
\right .
\end{equation}
Note that this definition makes sense since the cubes $\{Q_i\}$ are disjoint, so any $x\in \Omega_L$ belongs to only one $Q_i$. Also note that condition \eqref{eq:CZ3} implies that 
\begin{equation}\label{eq:gQ-bounded}
|g_Q(x)|\le 2^nL 
\end{equation}
for almost all $x\in Q$. The function $b_Q$ is determined by this choice of $g_Q$ as the difference 
\begin{equation*}
b_Q= \dfrac{f-f_Q}{{a_\varepsilon(Q)}} -g_Q,
\end{equation*}
but we  also have a representation as
\begin{equation}\label{eq:bQ}
b_Q(x)=\sum_i \left (f(x)-f_{Q_i}\right )\dfrac{1}{{a_\varepsilon(Q)}}\chi_{Q_i}(x)=\sum_i b_{Q_i},
\end{equation}
where $b_{Q_i}=(f(x)-f_{Q_i})\dfrac{1}{{a_\varepsilon(Q)}}\chi_{Q_i}(x)$.

Now we start with the estimation of the desired $L^p$ norm from \eqref{eq:a priori estimate}. Consider on $Q$ the measure $\mu$ defined by $d\mu=\frac{w\chi_Q}{w(Q)}$. Then, by the triangle inequality, we have

\begin{eqnarray*}
\left( \frac{1}{ w(Q)  } \int_{ Q }   \frac{|f -f_{Q}|^p}{a_{\varepsilon}(Q)^p}   \,wdx\right)^{\frac{1}{p} } & \le & \|g_Q\|_{L^p(\mu)}+\|b_Q\|_{L^p(\mu)}\\
& \le & 2^nL + \left (\dfrac{1}{w(Q)}\int_{\Omega_L}\sum_j |b_{Q_j}|^p \, w dx \right )^{1/p}\\
\end{eqnarray*}

Let us observe that the last integral of the sum, by the localization properties of the functions $b_{Q_i}$, can be controlled:
\begin{eqnarray*}
\int_{\Omega_L} |\sum_j b_{Q_j} |^p \, w dx  & \le & \sum_i\int_{Q_i}\left | b_{Q_j}\right |^p \, w dx\\
& = & \dfrac{1}{a_{\varepsilon}(Q)^p}\sum_i\frac{a_{\varepsilon}(Q_i)^pw(Q_i)}{w(Q_i)}\int_{Q_i}\left |\frac{f-f_{Q_i}}{a_{\varepsilon}(Q_i)}\right |^p \, w dx \\
& \le & \dfrac{X_{\varepsilon}^p }{a_{\varepsilon}(Q)^p}\sum_ia_{\varepsilon}(Q_i)^pw(Q_i),
\end{eqnarray*}
where $X_{\varepsilon}$ is the quantity defined by
$$
X_{\varepsilon}=\sup_Q \left (\frac{1}{w(Q)}\int_{Q}\left |\frac{f-f_{Q}}{a_{\varepsilon}(Q)}\right |^p \, w dx \right )^{1/p}.
$$
which is finite since $f$ is bounded and $a_{\varepsilon}(Q) >\varepsilon$. Then we obtain that
\begin{equation}\label{eq:UsingSmallness1}
\left( \frac{1}{ w(Q)  } \int_{ Q }  \frac{|f -f_{Q}|^p}{a_{\varepsilon}(Q)^p}  \,wdx\right)^{\frac{1}{p} }  \le  2^nL + X_{\varepsilon}\left ( \dfrac{\sum_ia_{\varepsilon}(Q_i)^pw(Q_i) }{a_{\varepsilon}(Q)^pw(Q)}\right )^{1/p}.
\end{equation}
Therefore, using the result from Lemma \ref{lem:SDps-a(Q)+eps} for the modified functional $a(Q)+\varepsilon$, it follows that
\begin{equation}\label{eq:UsingSmallness2}
\left( \frac{1}{ w(Q)  } \int_{ Q }  \frac{|f -f_{Q}|^p}{a_{\varepsilon}(Q)^p}  \,wdx\right)^{\frac{1}{p} } \le 2^n L +  X_{\varepsilon} \frac{C_{\|a\|}}{L^{1/\max\{s,r'\}}}.
\end{equation}
This holds for every cube $Q$, so taking the supremum we obtain
$$
X_{\varepsilon}\leq 2^n L +    X_{\varepsilon} \frac{C_{\|a\|}}{L^{1/\max\{s,r'\}}}.
$$
Choosing $L$ large enough independent of $\varepsilon$, it follows that
$$
X_{\varepsilon}\leq  c_{s,r',\|a\|}.
$$
 for any $\varepsilon$. By the monotone convergence theorem, we also conclude that \eqref{eq:a priori estimate} holds, namely

\begin{equation}\label{eq:X-finite}
X= \sup_Q \left( \frac{1}{ w(Q)  } \int_{ Q }   \frac{|f -f_{Q}|^p}{a(Q)^p}   \,wdx\right)^{\frac{1}{p} }<\infty
\end{equation}
Once we have proved that $X$ is finite, we can proceed to  the precise quantitative estimate \eqref{eq:FirstMainEstimate-Proof}. The steps are exactly the same as before, but using directly $X$ instead of the approximation $X_\varepsilon$ and therefore obtaining the exact same inequality \eqref{eq:UsingSmallness1} but with no $\varepsilon$. Since we are now dealing with the original functional $a$, we use the smallness preservation condition to obtain a better version of \eqref{eq:UsingSmallness2}, namely
$$
X\le 2^n L +  X \frac{\|a\|}{L^{1/s}}.
$$
Now we choose $L=2e\max\{\|a\|^s,1\}$ so the above inequality becomes
$$
X\le 2^n 2e\|a\|^s \left ((2e)^{1/s}\right )'\le e2^{n+1}s\|a\|^s,
$$
using the elementary fact that $\left ((2e)^{1/s}\right )'\le s$. This is the desired inequality \eqref{eq:FirstMainEstimate-Proof}:
$$
\left( \frac{1}{ w(Q)  } \int_{ Q } |f -f_{Q}|^p \,wdx\right)^{\frac{1}{p} }\le s c_n \|a\|^s a(Q)
$$
for any cube $Q$.

\end{proof}

We may now proceed to the proof of Corollary and \ref{cor:Poincare(p,p)-twoweight}. Once Theorem \ref{thm:Lp(w)-a(Q)-clean} is proved, the key step is to verify that the corresponding functionals satisfy the smallness preservation condition.

\begin{lemma}\label{lem:L-small}
Let $w$ be a weight, $L>1, 0< p<\infty$ and let $a(Q)$ defined as in \eqref{eq:general-a(Q)}. Then $a\in SD_p^{n/\alpha}(w)$.
\end{lemma}

\begin{proof}
Let $\{Q_i\}\in S(L)$. If $\frac{p\alpha}{n}<1$ use H\"older's inequality and convexity, 
\begin{eqnarray*}
\sum_{i}a(Q_i)^pw(Q_i) & \le & \sum_{i}\ell(Q_i)^{p\alpha}\mu(Q_i)\\
& = & \sum_i |Q_i|^{\frac{p\alpha}{n}}\mu(Q_i)\\
&\le & \left (\sum_i |Q_i|\right )^{\frac{p\alpha}{n}}\left (\sum_i \mu(Q_i)^{(\frac{n}{p\alpha})'}\right )^\frac{1}{(\frac{n}{p\alpha})'}\\
&\le & \left (\frac{|Q|}{L}\right )^{\frac{p\alpha}{n}}\sum_i \mu(Q_i)\\
& = & \left (\frac{1}{L}\right )^{\frac{p\alpha}{n}}\ell(Q)^{p\alpha}\mu(Q)=\left (\frac{1}{L}\right )^{\frac{p\alpha}{n}}a(Q)^pw(Q).
\end{eqnarray*}
If $\frac{p\alpha}{n}\geq 1$ use that $\mu(Q_i)\leq \mu(Q)$ and convexity. 
\end{proof}

Another interesting example is related to the classical unweighted Sobolev exponent $p^*$  given by 
$$
\frac{1}{n}=\frac{1}{p} -\frac{1}{p^*}\qquad    1\leq p<n.
$$
As before, consider the functional from \eqref{eq:general-a(Q)} but in the unweighted version and for $\alpha=1$:
\begin{equation*} 
a(Q)=\ell(Q)\left(\frac{1}{|Q|} \mu(Q) \right)^{1/p}. 
\end{equation*}
Then if $p\leq q<p^*$ we have that $a$ satisfies the $SD_s^{q}$ condition with $s>1$ given by 
$$
\frac{1}{s}=\frac{1}{q} -\frac{1}{p^*}.
$$
Indeed, this follows from H\"older's inequality applied 
to the exponent $r=\frac{p^*}{p^*-q}>1$ 
\begin{eqnarray*}
\sum_{i}a(Q_i)^q |Q_i| & = & 
\sum_{i}   |Q_i|^{1-\frac{q}{p^*}}       \mu(Q_i)^{ \frac{q}{p} }\\
&\le & \left (\sum_i |Q_i|\right )^{1-\frac{q}{p^*}} \left (\sum_i \mu(Q_i)^{  \frac{p^*}{p} }\right )^\frac{q}{p^*}\\
&\le & \left (  \frac{|Q|}{L}  \right )^{\frac{q}{s}} \left (\sum_i \mu(Q_i) \right )^\frac{q}{p}\\
&\le & \left (\frac{|Q|}{L}\right )^{\frac{q}{s}} \mu(Q)^\frac{q}{p} =\left (\frac{1}{L}\right )^{\frac{q}{s}}a(Q)^q \,|Q|.
\end{eqnarray*}

The observation above can be generalized as follows.  Given a functional $a$ we define the largest exponent for which a satisfies $D^{p^*}$ with bound $\|a\|\le 1$, namely 
$$
\sum_{i}a(Q_i)^{p^*} |Q_i| \leq  a(Q)^{p^*}|Q|.
$$
Then if $q<p^*$, $a$ satisfies the $SD_s^{q}$ condition with $s>1$ given by 
$$
\frac{1}{s}=\frac{1}{q} -\frac{1}{p^*}.
$$
Indeed, if $t=\frac{p^*}{q}$
\begin{eqnarray*}
\sum_{i}a(Q_i)^q |Q_i| & = & 
\sum_{i}a(Q_i)^q  \,|Q_i|^{\frac{1}{t}} \,|Q_i|^{\frac{1}{t'}}   \\
&\le & \left ( \sum_i a(Q_i)^{p^*}  |Q_i|  \right )^{\frac{1}{t}} \left (\sum_i |Q_i| \right )^\frac{1}{t'}\\
&\le & \left (  a(Q)^{p^*}|Q|  \right )^{\frac{1}{t}} \left ( \frac{|Q|}{L} \right )^\frac{1}{t'} \\
&\leq &   a(Q)^{q}|Q|   \left ( \frac{1}{L} \right )^\frac{q}{s}. \\
\end{eqnarray*}

With the previous estimates on the functional $a$ in hand, we can now present the proof of the corollaries.

The proof of the two weight analogue follows similar steps.

\begin{proof}[Proof of Corollary \ref{cor:Poincare(p,p)-twoweight}]
The proof follow from the well known $(1,1)$ Poincar\'e inequality 
$$
\frac{1}{|Q|}\int_Q |f-f_Q|\leq c\, \ell(Q)\frac{1}{|Q|}\int_Q |\nabla f(x)| dx.
$$
We refer to Theorem \ref{thm:equiv-weak-strong-1n'-pointiwise} in the Appendix (Section \ref{sec:appendix}). If we combine that inequality with the (two weight) $A_p$ condition property \eqref{two weight Ap charact}, we have that
$$
\frac{1}{|Q|}\int_Q |f-f_Q|\leq c\,[u,v]^{\frac{1}{p}}_{A_p}\ell(Q)\left (\frac{1}{u(Q)}\int_Q|\nabla f|^{p} \, v \ dx\right )^{1/p}.
$$
Now we just need to consider the functional 
$$
a(Q)=[u,v]^{\frac{1}{p}}_{A_p}\ell(Q)\left (\frac{1}{u(Q)}\int_Q|\nabla f|^{p} \, v \ dx\right )^{1/p}.
$$
and apply Theorem \ref{thm:Lp(w)-a(Q)-clean}. 
\end{proof}

\section{The Keith-Zhong phenomenon }\label{sec:KZ}

Here we present a result in the spirit of the work of Keith and Zhong on the open ended property of Poincar\'e inequalities.

\begin{proof} [Proof of Corollary \ref{cor:K-Z phenomenon}. ]

We have to prove that if $1\leq p<p_0$ and $w\in A_p$, then 
\begin{equation*}
\left (\frac{1}{w(Q)}  \int_Q|f-f_{Q}|^{p} \,w  \right )^\frac{1}{p}\leq
c\,\varphi( c_{p,p_0,n} [w]_{A_{p}}^{\frac{p_{0}-1}{p-1}} )\,  \ell(Q) \left(\frac{1}{w(Q)}   \int_Q g^{p}\,w \right)^{\frac{1}{p}}.
\end{equation*}
Let us define the functional
$$
a(Q) = \varphi([w]_{A_{p_0}})
 \ell(Q) \left(\frac{1}{w(Q)} \int_Q g^{p_0}\,wdx \right)^{1/p_0}.
$$
Then $a\in SD_{p_0}^{n}(w)$ by Lemma \ref{lem:L-small}. Hence, by Theorem \ref{thm:Lp(w)-a(Q)-clean} we have that 
\begin{equation*}
\left( \frac{1}{ w(Q)  } \int_{ Q }   |f -f_{Q}|^{p_0}   \,wdx\right)^{\frac{1}{p_0} }  \, \leq C_n\, a(Q).
\end{equation*}
We rewrite the last inequality as 
\begin{equation}\label{eq:PIp-p}
\int_{ Q }   |f -f_{Q}|^{p_0}     wdx   \leq C_n^{p_0}\,   \varphi([w]_{A_{p_0}})^{p_0}  \ell(Q)^{p_0} \int_Q g^{p_0}\,wdx.
\end{equation}

As already mentioned we use ideas from the theory of extrapolation with weights as can be found in \cite{CMP-Book}. Let $p \in (1,p_0)$  and let  $w\in A_p$. We will use the so called Rubio de Francia's algorithm. For any $h\in L^p$, we define 
\begin{equation*}
R(h)= \sum_{k=0}^\infty \frac1{2^k}\frac{M^k
(h)}{\|M\|_{L^{p}(w)}^k}.
\end{equation*} 

The operator $R$ satisfies the following three conditions.

(A) \quad $h\le R(h)$

\vspace{.2cm}

(B) \quad $\|R(h)\|_{L^{p}(w)}\le
2\,\|h\|_{L^{p}(w)}$

\vspace{.2cm}

(C) \quad  $[R(h)]_{A_{1}}\leq 2\,   \|M\|_{L^{p}(w) }$

Then we have, for some $\alpha>0$ to be chosen later, that
\begin{eqnarray*}
\left (  \int_Q|f-f_{Q}|^{p} \,wdx  \right )^\frac{1}{p} & = & \left (  \int_Q|f-f_{Q}|^{p}R(\chi_Qg)^{-\alpha p}R(\chi_Qg)^{\alpha p}   \,wdx \right )^\frac{1}{p}\\
&\le & I.II,
\end{eqnarray*}
by using H\"older's inequality with the pair $q=\frac{p_0}{p}>1$ and $q'=(\frac{p_0}{p})'=\frac{p_0}{p_0-p}$, where 
\begin{equation*}
I=\left (   \int_Q|f-f_{Q}|^{p_0}R(\chi_Qg)^{-\alpha p_0}  \,wdx  \right )^{1/p_0}
\end{equation*}
and 
\begin{equation*}
II=\left (   \int_Q R(\chi_Qg)^{\alpha p\left( \frac{p_0}{p}\right )'}   \,wdx  \right )^{\frac{1}{p\left( \frac{p_0}{p}\right )'}}.
\end{equation*}
We control the first term  $I$ by choosing $\alpha= \frac{p_0-p}{p_0}>0$ and defining 
$$v:=R(\chi_Qg)^{-\alpha p_0}=R(\chi_Qg)^{-(p_0-p)}.$$
Now we claim that  $v$ belongs to $A_{p_0}$ since 
$$
 [v]_{A_{p_0}}=[R(\chi_Qg)^{-(p_0-p)}\,w]_{A_{p_{0}}}
   \leq C\, \|M\|_{L^{p}(w) }^{p_{0}-p}\, [w]_{A_p}  \leq C [w]_{A_p}^ {\frac{p_0-p }{p-1}  }[w]_{A_p}.
$$
Indeed for any cube and by the definition of $A_{1}$, we have that 
$$
\avgint_Q  R(\chi_Qg)^{-(p_0-p)}\,wdx \le\, [R\chi_Qg]_{A_{1}}^{p_{0}-p}   \Big( \avgint_Q  R(\chi_Qg)\,dx  \Big)^{-(p_{0}-p)}\, 
\avgint_Q  \,wdx.
$$
by setting $q=\frac{p_0-1}{p_{0}-p}>1$ and observing that $q'=\frac{p_0-1}{p-1}$. Then, Holder's inequality yields
\begin{eqnarray*}
\avgint_Q   \Big(R(\chi_Qg)^{-(p_0-p)}\,w \Big)^{1-p'_{0}} \,dx 
& = &\avgint_Q  R(\chi_Qg)^{ \frac{p_{0}-p}{p_{0}-1}   }\,w^{1-p'_{0}} \,dx\\
&\le &
\Big( \avgint_Q  R(\chi_Qg)\,dx                        \Big)^{\frac{p_{0}-p}{p_{0}-1}}
\Big(\avgint_Q   w^{1-p'} \,dx \Big)^{\frac{p-1}{p_{0}-1}}.
\end{eqnarray*}
Hence, combining all previous estimates,  we have 
\begin{eqnarray*}
 [R(\chi_Qg)^{-(p_0-p)}\,w]_{A_{p_{0}}}  & \leq &  [R(\chi_Qg)]_{A_{1}}^{p_{0}-p}\, [w]_{A_p}\\
& \le & 2^{p_{0}-p}\, \|M\|_{L^p(w)}^{p_{0}-p}\, [w]_{A_{p}} \\
&  \leq  &c_{p,p_0,n} [w]_{A_{p}}^{\frac{p_{0}-p}{p-1}}\, [w]_{A_{p}}\\
&=& c_{p,p_0,n}\,  [w]_{A_{p}}^{\frac{p_{0}-1}{p-1}}.
\end{eqnarray*}%
by (C) above.

We can apply now \eqref{eq:PIp-p},
\begin{eqnarray*}
I & \le &  c\,\varphi( c_{p,p_0,n}\,  [w]_{A_{p}}^{\frac{p_{0}-1}{p-1}} ) \,  \ell(Q)\left (    \int_Qg^{p_0}\,R(\chi_Q g )^{-(p_0-p)} \,wdx  \right )^{1/p_0}\\
& \leq & c\,\varphi( c_{p,p_0,n}\,  [w]_{A_{p}}^{\frac{p_{0}-1}{p-1}}) \, \ell(Q)\left (    \int_Qg^{p} \,wdx\right )^{1/p_0},
\end{eqnarray*}
by property (A) above.

For the second factor $II$, note that by the choice of $\alpha$, we have that
$\alpha \left (\frac{p_0}{p}\right )'=\frac{p_0-p}{p_0} \frac{p_0}{p_0-p}=1$ and 
\begin{equation*}
II = 
\left (  \int_Q R(\chi_Q g )^{p} \,wdx  \right )^{\frac{1}{p\left( \frac{p_0}{p}\right )'}}
\le 2\,\left ( \int_Q g^{p} \,wdx  \right )^{\frac{p_0-p}{p_0p}}
\end{equation*}
by (B) above. Therefore, collecting estimates and noting that $\frac{1}{p}+\frac{p-p_0}{p_0p}=\frac{1}{p_0}$, we obtain
\begin{equation*}
\left (  \int_Q|f-f_{Q}|^{p} \,wdx  \right )^\frac{1}{p}  \leq C\ell(Q) \left (\int_Q g^{p} \,wdx\right )^{\frac{1}{p}}.
\end{equation*}

\end{proof}

\section{Poincar\'e-Sobolev type inequalities}\label{sec:PoincareSobolev}

In this section we will present the proof of Theorem \ref{thm:ptimes-Aq}. Here the key step to obtain meaningful Poincar\'e type inequalities is to find nontrivial examples of functionals $a$ and its corresponding indices $p_{w}^*$. To that end, we will consider functional of the form
$$
a(Q)=\ell(Q)\left(\frac{1}{w(Q)}\mu(Q) \right)^{1/p},
$$
and assume that the weight $w$ is in $A_q$ for some $1\le q\le p<n$. Note that this includes both endpoint cases $q=1$ and $q=p$. The idea is that we will use the fact of $w$ being an $A_p$ weight to build the specific functional for Poincar\'e inequalities. But assuming a stronger condition, namely $w\in A_q$, gives us better estimates for the index $p_{w}^*$.

The following definition proposes a suitable index $p^*$ associated to the values of $p,q$  related to the geometric properties of the weight $w$. In addition, we will consider an auxiliary parameter $M>1$ (that will also depend on $w$) to achieve the desired smallness preservation.

\begin{definition}\label{def:p*M}
Consider two indices $p,q$ such that $1\le q\le p < n$. For $M>1$ we define $p_M^*:=p(n,q,M)$ by the condition
\begin{equation}\label{eq:ptimesM}
\frac{1}{p} -\frac{1}{ p_M^*}=\frac{1}{nqM}.
\end{equation}
\end{definition}
Note that $p_M^*$ is smaller than the classical Sobolev exponent, namely the sharp one corresponding to the Lebsegue measure case:
\begin{equation*}
p<p_M^*<p^*=\frac{pn}{n-p} \qquad 1\leq p<n.
\end{equation*}
 
The next lemma contains the main estimate for a functional of the form \eqref{eq:general-a(Q)}  for $\alpha=1$.

\begin{lemma} \label{lem:smallAq} 

Let $1\le q\le p<n$,    and let $a(Q)$ defined as 

$$a(Q)=\ell(Q)\left(\frac{1}{w(Q)}\mu(Q) \right)^{1/p}.
$$
Let $w \in A_q$ and $p_M^*$ defined as in \eqref{eq:ptimesM}. Then for any family $\{Q_i\}$ of pairwise disjoint subcubes of $Q$ such that $\{Q_i\}\in S(L)$, $L>1$, the following inequality holds:
\begin{equation}\label{eq:Keyestimate}
\sum_{i}a(Q_i)^{ p^*_{M} }\,w(Q_i) \leq  [w]_{A_q}^{ \frac{ p^*_{M} }{nqM} }
\,\left (  \frac{1}{L}   \right )^{ \frac{ p^*_{M} }{nM'} } a(Q)^{ p^*_{M} } w(Q).
\end{equation}
This condition says that the functional $a$ ``preserves smallness'' for the exponent $p_M^*$ defined in \eqref{eq:ptimes-Aq} with index $nM'$ and constant $[w]_{A_q}^{ \frac{ 1 }{nqM} }$. That is, $a\in SD_{p^*_M}^{nM'}(w)$.
\end{lemma}

Before proceeding with the proof, we recall from \eqref{eq:prop-Ap-set} that for $A_q$ weights we have the geometric estimate
\begin{equation*}
\left (\frac{|E|}{|Q|}  \right )^{q}\leq [w]_{A_q}  \frac{w(E)}{w(Q)},
\end{equation*}
valid for any subset   $E\subset Q$.

\begin{proof}[Proof of Lemma \ref{lem:smallAq}]

Let $M>1$. For simplicity in the exposition, we will omit the subindex $M$ and just use $p^*$ instead of $p^*_M$. To verify the smallness preservation for the functional $a$, we compute
\begin{eqnarray*}
\sum_{i}a(Q_i)^{p^*}w(Q_i) & = & \sum_{i} \mu(Q_i)^{ \frac{p^*}{p} }    \left (    \frac{\ell(Q_i)}{w(Q_i)^{\frac{1}{p} -\frac{1}{p^*}} } \right )^{ p^*  }          \\
& = &  \sum_{i} \mu(Q_i)^{\frac{p^*}{p}}    \left(    \frac{ |Q_i| }{w(Q_i)^{ \frac{1}{qM}  } } \right )^{ \frac{p^*}{n}  } \\
& \leq &  [w]_{A_q}^{\frac{p^*}{nqM} }
\left( \frac{|Q|^{q}}{w(Q) }   \right)  ^{ \frac{p^*}{nqM}   }  \  
  \sum_{i} \mu(Q_i)^{\frac{p^*}{p}}    |Q_i|^{ \frac{p^* }{ n M'}   } \\
\end{eqnarray*}
Now, since $p<n$, $\frac{p^*}{nM'}<1$ and we use H\"older's inequality with $t=\frac{nM'}{p^* }>1$. Then
\begin{eqnarray*}
\sum_{i}a(Q_i)^{p^*}w(Q_i) &\le & 
\left( \frac{[w]_{A_q}|Q|^{q}}{w(Q) }   \right)  ^{ \frac{p^*}{nqM}   }    
\left (\sum_i  \mu(Q_i)^{  \frac{t'p^*}{p} }    \right )^\frac{1}{t'} \left (\sum_i |Q_i|\right )^{ \frac{p^*}{nM'}  }\\
&\leq &   \left( \frac{[w]_{A_q}|Q|^{q}}{w(Q) }   \right)  ^{ \frac{p^*}{nqM}   }        
 \left (\sum_i  \mu(Q_i)    
\right )^{ \frac{p^*}{p} }  \, |Q|^{ \frac{p^*}{nM'}  }\, \left (  \frac{1}{L}   \right )^{ \frac{p^*}{nM'}  }
\\
&= &  [w]_{A_q}^{\frac{p^*}{nqM} }
\frac{\ell(Q)^{p^*}}{w(Q)^{  \frac{p^*}{npM}} }
\mu(Q)^{ \frac{p^*}{p} }  \, \left (  \frac{1}{L}   \right )^{ \frac{p^*}{nM'}  }
\\
&= &  
[w]_{A_q}^{\frac{p^*}{nqM} }
a(Q)^{p^*}w(Q)\,\left (  \frac{1}{L}   \right )^{ \frac{p^*}{nM'}  }\\
\end{eqnarray*}
\end{proof}

\begin{proof} [Proof of Theorem \ref{thm:ptimes-Aq}] 
The proof of this theorem is just the proof of Theorem \ref{thm:Lp(w)-a(Q)-clean}  combined  with Lemma \ref{lem:smallAq} with a different choice of the parameters $L$ and $M$ there. More precisely, Lemma \ref{lem:smallAq} says that the functional $a$ satisfies the smallness preservation property for the index $p_M^*$ defined in \eqref{eq:ptimesM} with exponent $s=nM'$ and norm $\|a\|=[w]^{\frac{1}{nqM}}_{A_q}$.
If we choose $M=1+\log[w]^{\frac{1}{q}}_{A_q}$, then the proof of Theorem \ref{thm:Lp(w)-a(Q)-clean} produces the estimate
\begin{eqnarray*}
\left( \frac{1}{ w(Q)  } \int_{ Q }   |f -f_{Q}|^{p_M^*}    \,wdx\right)^{\frac{1}{p_M^*}}  & \leq  & C 2^{n+1} s\, \|a\|^{s}\, a(Q) \\
& \le & C2^{n+1} nM'[w]_{A_q}^{\frac{M'}{qM}} \, a(Q)\\
& \le & C2^{n+1} n \frac{1+\log[w]^{\frac{1}{q}}_{A_q}}{\log[w]^{\frac{1}{q}}_{A_q}}[w]_{A_q}^{\frac{1}{\log[w]_{A_q}}} \, a(Q)\\
& \le & C2^{n+2} n \, a(Q)
\end{eqnarray*}
where we used that $t^{\frac{1}{\log t}}=e$ and assumed that $[w]_{A_q}\geq e^q$. 
In the contrary we can use Theorem \ref{thm:general-a(Q)}  since the functional $a$ satisfies the $D_{p}(w)$ condition \ref{eq:Dp} with $p=p^*_M$ (this is nothing more than  Lemma \ref{lem:smallAq} with $L=1$) and since $w\in A_{\infty}$. Now, the bounds obtained in the proof of Theorem \ref{thm:general-a(Q)} are not precise but are given by $\varphi ([w]_{A_q})$ with increasing $\varphi$ and hence the result holds also in the case 
$[w]_{A_q}\leq e^q$.   More precisely, we use the fact that the $A_q$ condition is an open ended property: any $A_q$ weight is also an $A_{q-\varepsilon}$ weight for a small value of $\varepsilon>0$ depending on $w$. Hence, the functional $a$ will also satisfy a $D_{p^*_M +\delta}(w)$  for some $\delta>0$. Then, Theorem \ref{thm:general-a(Q)} provides a \emph{weak} estimate for the exponent $p^*_M +\delta>p^*_M$. This implies, by Kolmogorov's inequality, the desired \emph{strong} estimate.

\end{proof}

Now we can derive the result in Corollary \ref{cor:ptimes-Aq} as follows.

\begin{proof}[Proof of Corollary \ref{cor:ptimes-Aq}]
Let us start by using the unweighted $(1,1)$ Poincar\'e inequality as in Corollary \ref{cor:Poincare(p,p)-twoweight}. Since the weight is in $A_p$, we can build the functional $a$ in the same way and obtain the starting point 
\begin{equation}\label{eq:mixed-L1Lp-2}
\avgint_Q|f(x)-f_Q|dx\le C [w]^{\frac{1}{p}}_{A_p}\ell(Q)\left (\frac{1}{w(Q)}\int_Q |\nabla f(x)|^p w\ dx\right )^{1/p}.
\end{equation}
where
$$a(f,Q):= C [w]^{\frac{1}{p}}_{A_p}\ell(Q)\left (\frac{1}{w(Q)}\int_Q |\nabla f(x)|^p w\, dx\right )^{1/p}$$

At this point, we have already paid with the unavoidable quantity $[w]^{\frac{1}{p}}_{A_p}$  to build the functional $a$. Now, knowing something extra about the weight, that is, $w\in A_q$, allows us to reach a higher exponent. Theorem \ref{thm:ptimes-Aq} is applicable, so we obtain the desired inequality:
$$
\left (\frac{1}{w(Q)}\int_Q |f-f_Q|^{p_w^*}\, w\ dx\right )^{\frac{1}{p_w^*}} \leq [w]^{\frac{1}{p}}_{A_p}\ell(Q)\left (\frac{1}{w(Q)}\int_Q|\nabla f|^{p} \, w \ dx\right )^{1/p}.
$$

\end{proof}

\section{Using the  good-\texorpdfstring{$\lambda$}{goodlambda} }\label{sec:goodL}

In this section we will provide the proof of Theorem \ref{thm:goodL-aQ}.

\begin{proof}[Proof of Theorem   \ref{thm:goodL-aQ}]
Although we have introduced the notation $\|a\|_{D_r(w)}$ in the statement of this theorem, here the $SD_p^n$ norm of the functional $a$ is assumed to be equal to 1 and therefore there is no chance of confusion. We choose then to simplify the notation and simply write $\|a\|$ instead of $ \|a\|_{D_r(w)} $.

Now we proceed with the proff. Fix a cube $Q$, we have to prove that

\begin{equation}\label{eq:claim-weaknorm}
t^{r}\,
\frac{
w(
\{ x \in Q : |f(x)-f_{Q}|  > t\}
)
} { w(Q) }
\le
(c\|a\|)^{r}\,   [w]^{\frac{r}{p}}_{A_p}\, a(Q)^r
\end{equation}
with $c$ independent of $Q$, $t$ and $[w]_{A_{p}}$. 

Now, for each $t>0$, we let 
$$\Omega_{t} = \{ x \in Q : M(f-f_Q)(x) >
t\}$$
where $M$ will denote in this proof the dyadic
Hardy--Littlewood maximal function relative to $Q$. Then by the Lebesgue differentiation theorem
\[
\{ x \in Q : |f(x) -f_Q|  > t\} \subset \Omega_{t}.
\]
We will assume that $t >a(Q)$ since otherwise \eqref{eq:claim-weaknorm} is trivial. Hence
\[
t > a(Q)\ge
\frac{1}{|Q|}\int_{Q} |f-f_Q|
\]
and we can consider the Calder\'{o}n--Zygmund covering lemma of $|f-f_Q|$ relative to $Q$ for these values of $t$. This yields a collection $\{Q_{i}\}$ of dyadic subcubes of $Q$, maximal with
respect to inclusion, satisfying $ \Omega_{t}= \cup_{i}Q_{i}$ and
\begin{equation*}
t < \frac{ 1 }{ |Q_{i}| }
\int_{Q_{i}} |f-f_Q| \le 2^{n}\,t
\end{equation*}
for each $i$. Now let $q>1$ a big enough number that will be chosen in a
moment. Since $\Omega_{q\,t} \subset \Omega_{t}$, we have that
\begin{eqnarray*}
w(\Omega_{q\,t}) & = & w( \Omega_{q\,t} \cap \Omega_{t} ) \\
& =&  \sum_{i} w ( \{x\in Q_{i}: M(f-f_Q)(x)> q\,t \} )\\
&=&\sum_{i} w( \{ x \in Q_{i}: M( (f-f_Q) \chi_{ Q_{i} })(x) >qt \})
\end{eqnarray*}
where the last equation follows by the maximality of each of the cubes $Q_{i}$. Indeed, for any of these $i$'s and $x\in Q_i$, we have
\begin{eqnarray*}
M( |f -f_Q|\chi_{ Q })(x) & = &\max
\{
\sup_{
\stackrel{P: x \in P \in {\mathcal D}(Q) }{P \subseteq   Q_i} }\avgint_{P}|f-f_Q|,
\sup_{
\stackrel{P: x \in P \in {\mathcal D}(Q) }{P \supset Q_i } }\avgint_{P}|f-f_Q|
\}\\
& = & \sup_{
\stackrel{P:x \in P \in {\mathcal D}(Q) }{P \subset Q_i } }\avgint_{P}|f|\\
& = & 
M((f-f_Q)\chi_{Q_i})(x),
\end{eqnarray*}
since by the maximality of the cubes $Q_i$ when $P$ is dyadic (relative to $Q$) containing $Q_i$ then
\[
\frac{1}{|P|}\int_{P}|f-f_Q| \le t.
\]
On the other hand for arbitrary $x$,
\begin{eqnarray*}
|f(x)-f_Q| & \le & \absval{f(x)-f_{Q_{i}}} + \absval{f_Q-f_{Q_{i}}}\\
& \le & \absval{f(x)-f_{Q_{i}}} + \frac{1}{ |Q_{i}|} \int_{Q_{i}} |f-f_Q|\\
& \le & \absval{f(x)-f_{Q_{i}}} + 2^{n}\,t
\end{eqnarray*}
and then for $q=2^n+1$
\[
w(\Omega_{q\,t}) \le
\sum_{i} w(E_{Q_{i}}),
\]
where 
$$E_{Q_{i}} =\{x\in Q_{i}: M((f-f_{Q_{i}})\chi_{Q_{i}})(x)>t\}. 
$$

Let $\epsilon >0 $ to be chosen in a moment.  We split the family $\{Q_{i}\}$ in two sets of indices $I$ and $II$:

\begin{equation*}
i \in I \text{ if } a(Q_i) < \epsilon t \qquad\text{ and }\qquad i \in II  \text{ if } a(Q_i)\geq \epsilon t
\end{equation*}

Then
\[
w(\Omega_{q\,t}) \le
\sum_{i} w(E_{Q_{i}}) \le
\sum_{i\in I} w(E_{Q_{i}})+ \sum_{i\in II} w(E_{Q_{i}}) =I+II.
\]
Since $w\in A_p$ we use that $M$ is of weak type $(p,p)$ (with norm bounded by $[w]^{\frac1p}_{A_p}$) to control the the size of $E_{Q_i}$:
\begin{equation*} 
w(E_{Q_{i}}) \le \frac{ [w]_{A_p} }{t^p} \,
\frac{ 1 }{ w(Q_{i}) }
\int_{Q_{i}} \absval{ f-f_{ Q_{i}}  }^p\, wdx  \, w(Q_{i}).
\end{equation*}
And here is where the condition $(1)$ of the functional comes into play. Since we are assuming that $a$ preserves smallness, that is $a\in SD^{n}_p(w)$, we can apply Theorem \ref{thm:Lp(w)-a(Q)-clean} on the cube $Q_i$ to obtain 
\begin{equation}
w(E_{Q_{i}})   \leq c_n [w]_{A_p}\,\frac{a(Q_i)^p}{t^p}\,w(Q_{i}) \leq c_n [w]_{A_p}\,  \epsilon^p  \,w(Q_{i}),
\end{equation}
We conclude that 
\[
I \leq c_n[w]_{A_p}\,
\epsilon^p\,   w( \Omega_{t} ).
\]
For $II$ we use the hypothesis $(2)$ on $a$ in the theorem (which is simply condition $D_p(w)$ without smallness preservation). We have that
\begin{eqnarray*}
II & \le & \sum_{i \in II }w(Q_{i}) \le\sum_{i \in II }\left(\frac{1}{t\,\epsilon } a(Q_{i}) 
\right)^{r}w(Q_{i})\\
& \le & \frac{1}{ t^{r}\,\epsilon^{r} } \sum_{i } a(Q_{i})^{r}
w(Q_{i}) \leq \frac{\|a\|^r}{t^{r}\, \epsilon^{r} }
a(Q)^{r} w(Q).
\end{eqnarray*}
Combining all these estimates we have that
for $q=1+2^{n}$, 
\begin{equation*}
(qt)^{r}\, \frac{ w( \Omega _{qt} ) } { w(Q) }
\leq
q^r  t^{r}\, [w]_{A_p}\, \epsilon^p\,   \frac{ w( \Omega _{t} ) } { w(Q) }
+
\frac{ \|a\|^r q^{r} }{ \epsilon^{r}  }a(Q)^{r}    
\end{equation*}
for $t>a(Q)$. Observe that if we choose  $\epsilon \leq 1$, the same inequality holds for $t\le a(Q)$. Combining, we have the following inequality 
for $q=1+2^{n}$, $t>0$ and $0<\epsilon \leq 1$
\begin{equation*}
(qt)^{r}\, \frac{ w( \Omega _{qt} ) } { w(Q) }
\leq
q^r  t^{r}\, [w]_{A_p}\, \epsilon^p\,   \frac{ w( \Omega _{t} ) } { w(Q) }
+
\frac{ \|a\|^r q^{r} }{ \epsilon^{r}  }a(Q)^{r}    \quad t>0.
\end{equation*}
To conclude we use a standard good--$\lambda$ method. For $N>0$ we
let
\[
\varphi(N) =
\sup_{0<t<N} t^{r}\, \frac{ w( \Omega _{t} ) } { w(Q) },
\]
which is finite since is bounded by $N^{r}$. Since $\varphi$ is increasing we obtain that 
\[
\varphi(N)\le  \varphi(Nq)\leq
q^r\,      \, [w]_{A_{p}}\,\epsilon^p\,  \,  
 \varphi(N)
+
\frac{ \|a\|^r }{ \epsilon^{r}  }a(Q)^{r}.
\]
We now conclude by choosing $\epsilon$ such that
$$ q^r  \, [w]_{A_p}\, \epsilon^p = \frac12,  $$
which is equivalent to 
$$\frac{1}{\epsilon } =   2^{\frac1p} q^{\frac{r}{p}}\,  [w]_{A_p}^{\frac1p}.
$$
We conclude the proof by letting $N
\rightarrow \infty$.

\end{proof}

The proof of Lemma \ref{lem:new-smallAp} below provides a precise value of $r$ for the $D_r(w)$ condition satisfied for a functional $a$ of  the form \eqref{eq:general-a(Q)} with $\alpha=1$.

\begin{proof}    [Proof of Lemma   \ref{lem:new-smallAp}]

We will use again that if  $E\subset Q$ 
\begin{equation*}
\left (\frac{|E|}{|Q|}  \right )^{q}\leq[w]_{A_q}  \frac{w(E)}{w(Q)}.
\end{equation*}

As before,  for simplicity in the exposition, we will use $p^*$ instead of $p^*_w$.

Then, 
\begin{eqnarray*}
\sum_{i}a(Q_i)^{p^*  }w(Q_i) & = & \sum_{i} \mu(Q_i)^{ \frac{p^*  }{p} }    \left (    \frac{\ell(Q_i)}{w(Q_i)^{\frac{1}{p} -\frac{1}{p^*  }} } \right )^{ p^*    }          \\
& = & \sum_{i} \mu(Q_i)^{ \frac{p^*  }{p} }    \left (    \frac{ |Q_i|  ^{\frac1n}  }{w(Q_i)^{ \frac{1}{nq}  } } \right )^{ p^*    }        \\
& = &  \sum_{i} \mu(Q_i)^{\frac{p^*  }{p}}    \left(    \frac{ |Q_i| }{w(Q_i)^{ \frac{1}{q}  } } \right )^{ \frac{p^*   }{n}  } \\
& \leq &  [w]_{A_q}^{ \frac{ p^*  }{nq} }
\left( \frac{|Q|}{w(Q)^{\frac1q} }   \right)  ^{  \frac{p^*  }{n}   }  
  \sum_{i} \mu(Q_i)^{\frac{p^*  }{p}}     \\
& \leq &  [w]_{A_q}^{ \frac{ p^*}{nq} }
\left( \frac{|Q|}{w(Q)^{\frac1q} }   \right)  ^{  \frac{p^*  }{n}   }  
  \mu(Q)^{\frac{p^*  }{p}}     \\
&= &  
[w]_{A_q}^{ \frac{p^*  }{nq} }\,
a(Q)^{p^*  }w(Q)    \\
\end{eqnarray*}
and hence \,$\|a\|\leq [w]_{A_q}^{ \frac{1}{nq} }$

\end{proof}

\begin{proof}[Proof of Corollary \ref{cor:MainCoro}] Let us consider again inequality \eqref{eq:mixed-L1Lp-2} as a  starting point. Then we have that  
\begin{equation*}
\avgint_Q|f(x)-f_Q|dx\le C [w]^{\frac{1}{p}}_{A_p}\ell(Q)\left (\frac{1}{w(Q)}\int_Q |\nabla f(x)|^p w\ dx\right )^{1/p}.
\end{equation*}
Define the functional 
$$a(f,Q):= C [w]^{\frac{1}{p}}_{A_p}\ell(Q)\left (\frac{1}{w(Q)}\int_Q |\nabla f(x)|^p w\, dx\right )^{1/p}.$$
By Lemmas \ref{lem:L-small} and \ref{lem:new-smallAp}, we have conditions (1) and (2) from Theorem \ref{thm:goodL-aQ}. Taking into account the value of $\|a\|$ computed above, we obtain the desired estimate by using the  truncation method.
\end{proof}

\section{A mixed Poincar\'e inequality and applications to \texorpdfstring{$A_1$}{A1} weights}\label{sec:mixed}

We present here a different approach to the problem of weighted Poincar\'e inequalities and present the proof of Theorem \ref{thm:PoincareSobolev-2weights} and its consequences. 
We will use the following very simple lemma.

\begin{lemma}\label{lem:vanish}
Let $\mu$ be a finite measure such that $\text{supp}(\mu)\subset \Omega\subset \mathbb{R}^n$.  Consider a subset $E\subset \Omega$ such that $\mu(E)\ge \lambda\mu(\Omega)$ for some $\lambda \in (0,1)$ and a function $f$ vanishing on $E$. Then, for any constant $a\in \mathbb{R}$, we have that 
\begin{equation}\label{eq:vanish}
\|a\|_{L^q_\mu}\le \frac{1}{\lambda^q}\|f-a\|_{L^q_\mu}
\end{equation}
\end{lemma}
\begin{proof}
A straightforward computation shows that

\begin{eqnarray*}
\|a\|_{L_\mu^q}&=& \mu(\Omega)^\frac{1}{q}|a|\\
&\le & (\frac{1}{\lambda})^{1/q}\mu(E)^\frac{1}{q}\frac{1}{\mu(E)}\int_E|f(x)-a|  \ d\mu\\
& \le &  (\frac{1}{\lambda})^{1/q}  \mu(E)^\frac{1}{q}\frac{1}{\mu(E)}\left(\int_E |f(x)-a|^q \ d\mu \right)^\frac{1}{q}\mu(E)^\frac{1}{q'}\\
& \le &   (\frac{1}{\lambda})^{1/q}  \|f-a\|_{L^q_\mu}
\end{eqnarray*}
\end{proof}

Now we are ready to prove  the precise weighted Poincar\'e-Sobolev inequality from Theorem \ref{thm:PoincareSobolev-2weights}.

\begin{proof}[Proof of Theorem \ref{thm:PoincareSobolev-2weights}]
The goal is to prove  inequality \eqref{eq:P(p,p*)-local-avg}: 
\begin{equation*}
\left (\int_Q|f-f_{Q,w}|^{p^*}\ wdx \right )^{\frac{1}{p*}} \le C \left (\int_Q |\nabla f|^p \left (\frac{(M^c(w\chi_Q))^\frac{1}{n'}}{w}\right )^pw\ dx \right )^{\frac{1}{p}}.
\end{equation*} 

The proof of this inequality is based on the following interesting argument used in \cite{CW,DMRT} in a different context. We will use a relatively known equivalence between $(1,1)$ Poincar\'e inequalities and fractional integrals. More precisely, since the $(1,1)$ Poincar\'e inequality is true, the by  Theorem \ref{thm:equiv-weak-strong-1n'-pointiwise} in the Appendix we have that 
\begin{equation*}
\left (\int_Q |f(x)-f_Q|^{n'}\ d\mu\right )^\frac{1}{n'}\le C \int_Q |\nabla f(y)| (M\mu(y))^\frac{1}{n'}\ dy.
\end{equation*} 
We start by showing that this inequality implies a similar inequality without the average $f_Q$ when restricted to certain class of functions with a vanishing condition.

{\bf Claim: } Let $E\subset Q$ be any measurable subset of the cube such that $w(E)\ge \frac{w(Q)}{2}$.  For  any function $f$ vanishing on $E$, we have that 
\begin{equation}\label{eq:P(1,n')-local-NOavg}
\left (\int_Q|f|^{n'}w\ dx\right )^\frac{1}{n'}\le C \int_Q |\nabla f(y)| (Mw(y))^\frac{1}{n'}\ dy.
\end{equation} 

Let us define the measure $\mu=w \chi_Q dx$. Then we have that
\begin{equation*}
\|f\|_{L_\mu^{n'}}\le \|f-f_Q\|_{L_\mu^{n'}}+\|f_Q\|_{L_\mu^{n'}}\le (1+2^\frac{1}{n'})\|f-f_Q\|_{L_\mu^{n'}}
\end{equation*}
by a direct application of Lemma \ref{lem:vanish} with $a=f_Q$. Now we apply inequality \eqref{eq:P(1,n')-mu-local-avg} from Corollary \ref{cor:P(1,n')-mu-local-avg} to obtain
\begin{equation*}
\|f\|_{L_\mu^{n'}}\lesssim\|f-f_Q\|_{L_\mu^{n'}}\lesssim  \int_Q |\nabla f(y)| (M\mu(y))^\frac{1}{n'}\ dy
\end{equation*}
prove the claim and conclude that  inequality \eqref{eq:P(1,n')-local-NOavg} holds.

Now we proceed to an intermediate Poincar\'e-Sobolev inequality with the pair of exponents $(p^*,p)$. Let us denote by $g_-$ and $g_+$ the positive and negative parts of a measurable function $g$. Since we are integrating over a cube we can use that, for a given function $f$, there is a real number $\lambda$ such that 
\begin{equation}\label{eq:split}
\int_Q(f(x)-\lambda)_+^{p*} \ dx = \int_Q(f(x)-\lambda)_-^{p*} \ dx 
\end{equation}
We also have (for the same $\lambda$) that for any $q>1$,
$$Q=\{x\in Q: (f(x)-\lambda)^q_+=0\}\cup \{x\in Q: (f(x)-\lambda)^q_-=0\}.$$

We can assume that $w( \{x\in Q: (f(x)-\lambda)^q_+=0\})\ge w(Q)/2$ and then the function $(f-\lambda)_+^q$ satisfies the hypothesis from the claim above. Choosing $q>1$ such that $q=\frac{p*}{n'}=\frac{p(n-1)}{n-p}$, we apply \eqref{eq:P(1,n')-local-NOavg} to obtain
\begin{eqnarray*}
\left (\int_Q(f-\lambda)_+^{p*} w \right) ^\frac{1}{n'} &\lesssim & \int_Q|f-\lambda|^{q-1} |\nabla f|(Mw)^\frac{1}{n'} \\
& \le & \int_Q|f-\lambda|^{q-1} w^\frac{1}{p'}|\nabla f|\frac{(Mw)^\frac{1}{n'}}{w}w^\frac{1}{p} \\
& \le & \left (\int_Q|f-\lambda|^{(q-1)p'}w\right )^\frac{1}{p'}\left (\int_Q|\nabla f|^p\frac{(M^cw)^\frac{p}{n'}}{w^p}w \right)^\frac{1}{p}
\end{eqnarray*}
Since $(q-1)p'=\frac{n(p-1)p}{(n-p)(p-1)}=p*$ and $\frac{1}{n'}-\frac{1}{p'}=\frac{1}{p*}$, we obtain that
\begin{equation*}
\left (\int_Q(f(x)-\lambda)_+^{p*} w(x) \ dx\right) ^\frac{1}{p^*}\lesssim \left (\int_Q|\nabla f(x)|^p\left (\frac{(M^cw)^\frac{1}{n'}}{w}\right )^pw \ dx\right) ^\frac{1}{p}
\end{equation*}
By the relation in \eqref{eq:split}, the same inequality holds for the negative part $(f-\lambda)_-$, so we obtain the estimate
\begin{equation}\label{eq:f-lambda}
\|f-\lambda\|_{L_\mu^{p*}}\lesssim \|\nabla f\|_{L^p_v}
\end{equation}
where $v=\left (\frac{(M^cw)^\frac{1}{n'}}{w}\right )^pw$ and $\mu=w\chi_Q dx$. We also have, by Jensen's inequality, that
\begin{eqnarray*}
\|\lambda-f_{Q,w}\|_{L_\mu^{p*}} & = &  w(Q)^{1/p*}|\lambda- f_{Q,w}|\\
& \le & w(Q)^{1/p*} \int_Q |\lambda-f(x)| \frac{w(x)}{w(Q)}\ dx\\
& \le & w(Q)^{1/p*} \left (\int_Q |\lambda-f(x)|^{p^*} \frac{w(x)}{w(Q)}\ dx\right )^\frac{1}{p^*}\\
& \le &  \|f-\lambda\|_{L_\mu^{p*}}
\end{eqnarray*}
Collecting all previous estimates, we obtain the desired result. Note that 
\begin{equation*}
\|f-f_{Q,w}\|_{L_\mu^{p*}}\le \|f-\lambda\|_{L_\mu^{p*}}+\|\lambda-f_{Q,w}\|_{L_\mu^{p*}}\lesssim 2\|f-\lambda\|_{L_\mu^{p*}}.
\end{equation*}
We complete the proof by applying \eqref{eq:f-lambda} to obtain the desired estimate \eqref{eq:P(p,p*)-local-avg}:
\begin{equation*}
\left (\int_{Q}|f-f_{Q,w}|^{p^*}\ wdx\right )^{p*}\le C \int_Q |\nabla f|^p \left (\frac{(M^cw)^\frac{1}{n'}}{w}\right )^pw\ dy.
\end{equation*}

\end{proof}

\subsection{The case of \texorpdfstring{$A_1$}{A1} weights}

Now we can derive as a consequence the proof of the $A_1$ result.

\begin{proof}[Proof of Corollary \ref{thm:P(p,p*)-local-avg-A1}]
We start from \eqref{eq:P(p,p*)-local-avg} proved above using again that $v=\left (\frac{(M^cw)^{1/n'}}{w}\right )^pw$ to rewrite it as:
\begin{eqnarray*}
\left (\frac{1}{w(Q)}\int_Q|f-f_{Q,w}|^{p*}w\right )^\frac{1}{p*} & \lesssim & \frac{w(Q)^{1/p}}{w(Q)^{1/p*}} \left (\frac{1}{w(Q)}\int_Q |\nabla f|^p v\right )^\frac{1}{p}\\
& \lesssim & \left ( \frac{w(Q)}{|Q|}\right )^{1/n}|Q|^{1/n}\left (\frac{1}{w(Q)}\int_Q |\nabla f|^p v\right )^\frac{1}{p}\\
\end{eqnarray*}
since $\frac{1}{p}-\frac{1}{p^*}=\frac{1}{n}$. Then, we can control the average by using the maximal function to obtain that
\begin{eqnarray*}
\left (\frac{1}{w(Q)}\int_Q|f-f_{Q,w}|^{p*}w\right )^\frac{1}{p*} & \lesssim & \inf_{x\in Q}\left (Mw(x)\right )^\frac{1}{n} \ell(Q)\left (\frac{1}{w(Q)}\int_Q |\nabla f|^p v\right )^\frac{1}{p}\\
& \lesssim & \ell(Q)\left (\frac{1}{w(Q)}\int_Q |\nabla f|^p \left (\frac{Mw}{w}\right )^pw\right )^\frac{1}{p}\\
& \lesssim & [w]_{A_1}\ell(Q)\left (\frac{1}{w(Q)}\int_Q |\nabla f|^p w\right )^\frac{1}{p}.\\
\end{eqnarray*}
The proof is complete.
\end{proof}

\begin{remark}
We remark that in this case, we can see that the measure $w\ dx$ behaves like the Lebesgue measure at least for the range of local integrability on the left hand side.
\end{remark}

The next question here is regarding the sharpness of the exponent on the $A_1$ constant in the above inequality. We have the following result related to that issue. 

\begin{proposition}\label{pro:LowerBoundBeta}
Suppose that inequality \eqref{eq:P(p,p*)-local-avg-A1} holds with some power on the $A_1$ constant, namely
\begin{equation}\label{eq:P(p,p*)-local-avg-A1-beta}
\left (\frac{1}{w(Q)}\int_Q|f-f_{Q,w}|^{p*}w\right )^\frac{1}{p*}\le C [w]^\beta_{A_1}\ell(Q)\left (\frac{1}{w(Q)}\int_Q |\nabla f|^p w\right )^\frac{1}{p}.
\end{equation} 
Then $\beta\ge\frac{1}{p}$.
\end{proposition}
\begin{proof}
The conclusion follows from the analysis of an specific example. Consider the cube in $\mathbb{R}^n$, $n\ge 2$,  defined as $Q=(-1,1)^n$. For $\delta\in (0,1)$, define the weight $w$ by the formula $w(x):=|x|^{\delta-n}$. Then $w$ satisfies that $[w]_{A_1}\sim \frac{1}{\delta}$ and $w(Q)\sim \frac{1}{\delta}$. Now fix $0<\varepsilon<1/2$ and define the set $E:= Q\setminus (-2\varepsilon,2\varepsilon)^n$. Define on $Q$ a  piecewise affine Lipschitz  function $f$ such that $f(x)=0$ for all $x\in E$ and $f(x)=1$ on $(-\varepsilon,\varepsilon)^n$.

Inequality \eqref{eq:P(p,p*)-local-avg-A1-beta} implies, by Lemma \ref{lem:vanish}, that the same inequality holds without the average:
\begin{equation*}
\left (\frac{1}{w(Q)}\int_Q|f|^{p*}w\right )^\frac{1}{p*}\le  [w]^\beta_{A_1}\ell(Q)\left (\frac{1}{w(Q)}\int_Q |\nabla f|^p w\right )^\frac{1}{p}.
\end{equation*} 

Now, a simple computation on the LHS shows that
\begin{equation*}
\left (\frac{1}{w(Q)}\int_Q|f|^{p*}w\right )^\frac{1}{p*}\gtrsim \left (\delta \int_{(-\varepsilon,\varepsilon)^n}|x|^{\delta-n}dx\right )^\frac{1}{p*}\gtrsim \varepsilon^{\frac{\delta}{p*}} 
\end{equation*} 
On the other hand, the RHS involving the gradient can be controlled as follows
\begin{eqnarray*}
[w]^\beta_{A_1}\ell(Q)\left (\frac{1}{w(Q)}\int_Q |\nabla f|^p w\right )^\frac{1}{p}
& \lesssim &
\delta^{-\beta+\frac{1}{p}}\left (\int_{(-2\varepsilon,2\varepsilon)^n\setminus (-\varepsilon,\varepsilon)^n}\frac{ |x|^{\delta-n}}{\varepsilon^p}dx\right )^\frac{1}{p}\\
& \lesssim &
\frac{\varepsilon^{\frac{\delta}{p}-1}}{\delta^{\beta-\frac{1}{p}}}
\end{eqnarray*} 
Therefore, inequality \eqref{eq:P(p,p*)-local-avg-A1-beta} would imply that 
\begin{equation*}
\varepsilon^{\frac{\delta}{p*}}  \lesssim \frac{\varepsilon^{\frac{\delta}{p}-1}}{\delta^{\beta-\frac{1}{p}}}
\end{equation*} 
For a fixed $\varepsilon$, this forces the condition $\beta\ge 1/p$.
\end{proof}

\section{Generalized Poincar\'e  inequalities with polynomials}\label{sec:higher}

The purpose of this  section  is to prove  Theorem  \ref{thm:higher-a(Q)} related to the analysis of Poincar\'e inequalities involving higher order derivatives.

A somewhat less-known result that we can use as a starting point, analogously to \eqref{eq:Poincare-L1}, is the following higher order inequality. There is a constant $C>0$ such that for any cube $Q$,

\begin{equation}
\frac{1}{ |Q| }\int_{Q} \absval{f(y)- \pi_Q(y)}\,dy \le C\,
\frac{ \ell(Q)^{m} }{  |Q| }
\int_{Q} \absval{\nabla ^{m}f} \,dy,
\label{higherorderpoincare}
\end{equation}
for some polynomial $\pi_Q$ depending on $f$ and $Q$ of degree at most
$m-1$, where $m$ is a positive integer. Here $\nabla ^{m}f=
\{D^{\sigma}f \}_{ |\sigma|=m} $ and $\absval{\nabla 
^{m}f} =  \sum \absval{ D^{\sigma}f }$. Estimates of this type can be
found for example in \cite{Bo}.

Note that for this kind of inequalities involving higher order derivatives, we cannot make use of the truncation method as in Proposition \ref{pro:P(p,p)-I1vsM} or Proposition \ref{pro:two-weights}. We present in this section a variation of Theorem \ref{thm:Lp(w)-a(Q)-clean} adapted to the oscillation with respect to ``optimal polynomials''. 

Before we present the main result of this section, a few words on optimal polynomials are needed. We borrow the following definitions and properties from \cite{FPW98} which are based on \cite{DS}. Given a cube $Q\subset \mathbb{R}^n$ and an integer $m\ge 0$, we consider the space $\mathcal{P}_m$ of polynomials of degree at most $m$ in $n$ variables endowed with the inner product given by 
$$
<f,g>_Q:=\avgint_Q fg dx.
$$
There is an orthonormal basis with respect to this inner product that we will denote by $\{\phi_\alpha\}$, being $\alpha=(\alpha_1,\dots,\alpha_n)$ a multiindex of non negative integers such that $|\alpha|=\alpha_1+\cdots+\alpha_n\le m$. An important feature is that 
\begin{equation}\label{eq:equiv-norm-Pm}
\|\phi_r\|_{L^\infty}\le C\left (\frac{1}{|Q|}\int_Q |\phi_r|^2 dx\right) ^{1/2}=C,
\end{equation}
since the space $\mathcal{P}_m$ is finite dimensional and therefore all norms are equivalent. Let $P_Q$ the projection defined by the formula
$$
P_Q(f)=\sum_r \left (\frac{1}{|Q|}\int_Q f\phi_r dx\right )\phi_r.
$$
We clearly have from \eqref{eq:equiv-norm-Pm} that 
\begin{equation}\label{eq:Linfty-PQ}
\|P_Qf\|_{L^\infty}\le N C^2 \avgint_Q |f|,
\end{equation}
where $N$ depends on $m$. Moreover, as it is the case when $m=0$ and the projection is over the constants, we have the following optimality property:
\begin{equation*}
\inf_{\pi\in\mathcal{P}_m}\left (\avgint_Q |f-\pi|^p\right) ^{1/p} \approx \left (\avgint_Q|f-P_Qf|^p \right )^{1/p}.
\end{equation*}

We now proceed to present the proof of the announced result.

\begin{proof}[Proof of Theorem \ref{thm:higher-a(Q)}] 
The proof relies on a Calder\'on-Zygmund decomposition adapted to the function $\frac{|f-P_Q(f)|}{a(Q)}$ similar to the technique used in Theorem \ref{thm:Lp(w)-a(Q)-clean}. The difference here is that we do not decompose the function into the ``good'' and ``bad'' parts $g_Q$ and $b_Q$. We only work with the decomposition of the level set. 

More precisely, for a given $L>1$, we decompose the cube $Q$ into a family of dyadic subcubes maximal with respect to the inclusion satisfying inequalities similar to \eqref{eq:CZ1}, \eqref{eq:CZ2} and \eqref{eq:CZ3}, namely

\begin{equation}\label{eq:Poly-CZ1}
L  < \frac{1}{|Q_{j}|  }\int_{Q_{j}} \frac{|f-P_{Q}f|}{a(Q)} \, dy. 
\end{equation}
Then, if $P$ is dyadic with $P \supset Q_j$ 
\begin{equation}\label{eq:Poly-CZ2}
\frac{ 1 }{|P|  }
\int_{P} \frac{|f-P_{Q}f|}{a(Q)} \, dy   \leq L
\end{equation}
and hence 
\begin{equation}\label{eq:Poly-CZ3}
L  < \frac{ 1 }{|Q_{j}|  }
\int_{Q_{j}} \frac{ |f-P_{Q}f| }{a(Q)} \, dy \leq L\,2^{n}
\end{equation}
for each integer $j$.  Also note that 
$$  \left \{x\in Q: M_Q^d\left ( \frac{|f-P_{Q}f|}{a(Q)}\chi_{  Q }\right )(x) > L   \right \} = \bigcup_{j}Q_j=:\Omega_L.
$$
Then, by the Lebesgue differentiation theorem it follows that
$$\frac{|f(x)-P_{Q}f|}{a(Q)} \leq L   \qquad              a.e. \ x \notin \bigcup_{j}Q_j.  
$$
As before, we conclude that the collection of maximal cubes is $L$-small according to Definition \ref{def:L-small}, since we have that
\begin{equation*}\label{eq:Poly-CZ4a}
|\bigcup_{j}Q_j | < \frac{|Q|}{L}.
\end{equation*}
We start by splitting the integral 

\begin{eqnarray*}
\left( \frac{1}{ w(Q)  } \int_{ Q }   \frac{ |f -P_{Q}f|^p}{ a(Q)^p}  wdx \right)^{\frac{1}{p}} & \le & \left(
\frac{1}{ w(Q)  } \int_{Q \setminus  \Omega_L} \frac{ |f -P_{Q}f|^{p}}  { a(Q)^p  }      wdx  \right)^{\frac{1}{p}}
\\
&  & +\left(
\frac{1}{ w(Q)  } \int_{ \Omega_L} \frac{ |f -P_{Q}f|^{p}}  { a(Q)^p  }      wdx  \right)^{\frac{1}{p}}\\
&=& I+II
\end{eqnarray*}
By the Lebesgue differentiation theorem, we have that $I \leq L$. For $II$, we have that 
\begin{equation*}
\int_{\Omega_L}  \frac{|f -P_{Q}f|^{p }}{a(Q)^p} wdx = \sum_j   	\int_{Q_j} \frac{|f -P_{Q}f|^{p }}{a(Q)^p} wdx.
\end{equation*}
For each $j$, let us split again the integral by introducing the projection over the smaller cube $Q_j$ to obtain
\begin{equation*}
\int_{Q_j}  |f -P_{Q}f|^{p } wdx 
\le 2^{p-1}\left (    \int_{Q_j}  |f -P_{Q_j}f|^{p } wdx +    \int_{Q_j}  |P_{Q_j}f-P_Qf|^{p } wdx\right )
\end{equation*}

Now, since $P_{Q_j}$ is a projection, we have that $P_{Q_j}(P_Q f)=P_Qf$ and therefore we can compute for any $x\in Q_j$ that 
\begin{equation*}
|P_{Q_j}f-P_Qf|(x)\le |P_{Q_j}(f-P_Qf)|(x)\le C^2 \avgint_{Q_j}|f-P_Qf|\le C^2 2^nLa(Q)
\end{equation*} 
 by \eqref{eq:Linfty-PQ} and \eqref{eq:Poly-CZ3}. Therefore,
 $$
 \int_{Q_j}  |P_{Q_j}f-P_Qf|^{p } wdx\le C^{2p}2^{np}L^p a(Q)^p w(Q_j).
 $$
We define again the quantity \footnote{ Remark: Here we omit the details of proving that $X$ is finite. The interested reader may check that the arguments from the proof of Theorem \ref{thm:Lp(w)-a(Q)-clean} can be used here as well.  }
 $$
X=\sup_Q \left (\frac{1}{w(Q)}\int_{Q}\left |\frac{f-P_{Q}f}{a(Q)}\right |^p \, w dx \right )^{1/p}.
$$

After collecting all previous estimates and using the above definition, we obtain
\begin{eqnarray*}
\int_{\Omega_L}  \frac{|f -P_{Q}f|^{p }}{a(Q)^p} wdx  & \le  &2^{p-1}\sum_j \frac{a(Q_j)^pw(Q_j)}{a(Q)^pw(Q_j)}  	\int_{Q_j} \frac{|f -P_{Q_j}f|^{p }}{a(Q_j)^p} wdx  \\
&& + 2^{p-1+pn}C^{2p}L^p \sum_j w(Q_j)\\
& \le & 2^{p-1}X^p \frac{\sum_j   a(Q_j)^{p} \,w(Q_j)  }{ a(Q)^{p} } 
  +2^{p-1+pn}C^{2p}L^pw(Q)
\end{eqnarray*}
Therefore, using the smallness preservation of the functional $a$, we obtain
\begin{eqnarray*}
\left (\frac{1}{w(Q)}\int_{\Omega_L}  \frac{|f -P_{Q}f|^{p }}{a(Q)^{p}}  wdx \right )^\frac{1}{p}& \le & 2^{\frac{1}{p'}}C^2X\left( \frac{\sum_j   a(Q_j)^{p} \,w(Q_j)  }{ a(Q)^{p}   w(Q)  } 
\right)^{\frac{1}{p}}\\
&& +2^{\frac{1}{p'}+n}C^2L\\
&\le & 2^{\frac{1}{p'}}X \frac{\|a\|}{L^{1/s}}+2^{\frac{1}{p'}+n}C^2L.
\end{eqnarray*}
We can compute the integral over the whole cube $Q$:
\begin{eqnarray*}
\left( \frac{1}{ w(Q)  } \int_{ Q }   \frac{ |f -P_{Q}f|^p}{ a(Q)^p}  wdx \right)^{\frac{1}{p}} & \le & L+2^{\frac{1}{p'}}X \frac{\|a\|}{L^{1/s}}+2^{\frac{1}{p'}+n}C^2L\\
& \le & X \frac{2^{\frac{1}{p'}}\|a\|}{L^{1/s}}+(2^{\frac{1}{p'}+n}C^2+1)L.
\end{eqnarray*}
Now we proceed as in the proof of Theorem \ref{thm:Lp(w)-a(Q)-clean}. Taking the supremum, we obtain that
$$ 
X \left (1-\frac{2^{\frac{1}{p'}}\|a\|}{L^{1/s}}\right ) \le (2^{\frac{1}{p'}+n}C^2+1)L.
$$
Choosing $\frac{2^{\frac{1}{p'}}\|a\|}{L^{1/s}}=\frac{1}{(2e)^{1/s}}$ or equivalently $L=2e2^\frac{s}{p'}\|a\|^s$ we conclude that
\begin{eqnarray*}
X & \le & \left ((2e)^{1/s}\right )'(2^{\frac{1}{p'}+n}C^2+1)2e2^\frac{s}{p'}\|a\|\\
& \le & C_n s 2^\frac{s+1}{p'}\|a\|^s.
\end{eqnarray*}

\end{proof}

\section{Poincar\'e inequalities and  Lorentz spaces}\label{sec:Lorentz}

In this section we will present the proof of Corollary \ref{cor:Lp1(w)-a(Q)-clean} as another application of Theorem \ref{thm:Lp(w)-a(Q)-clean}  within the scale of Lorentz spaces where
the $A_{p,1}$ class of weights (see Section \ref{sec:prelim} for the precise definition)  plays a role. Indeed, there is an inequality for this class of weights very similar to \eqref{eq:prop-Ap-function} at the scale of Lorentz norms. In fact, we claim that  for a constant depending on $p$

\begin{equation}\label{eq:prop-Ap1-function}
\avgint_Q  g dx \leq c\,[w]_{A_{p,1}}^{\frac1p} \,\Big\| g \Big\|_{L^{p, 1}(Q, \frac{wdx}{w(Q)}  ) }\qquad  g\geq 0.
\end{equation}
To prove the claim we use Holder's inequality in the context of Lorentz spaces together with the $A_{p,1}$ condition \eqref{Ap,1}
\begin{eqnarray*}
\avgint_Q  g w^{-1}\,wdx & \leq &   c \Big\| g \Big\|_{L^{p, 1}(Q, \frac{wdx}{|Q|} ) }  \, \Big\| \frac{1}{w}\Big\|_{L^{p',\infty}(Q, \frac{wdx}{|Q|} ) }\\
& = &  c\Big\| g \Big\|_{L^{p, 1}(Q, \frac{wdx}{|Q|} ) }  \, \Big\| \frac{1}{w}\Big\|_{L^{p',\infty}(Q, \frac{wdx}{|Q|} ) }
\left(  \avgint_Q  w \right)^{\frac1p}
\left(  \frac{|Q|}{w(Q)} \right)^{\frac1p}\\
& \leq & c[w]_{A_{p,1}}^{\frac1p} \, \,\Big\| g \Big\|_{L^{p, 1}(Q, \frac{wdx}{w(Q)}  ) }
\end{eqnarray*}

Now, for the proof of Corollary  \ref{cor:Lp1(w)-a(Q)-clean} consider the following $L^1$ generalized unweighted Poincar\'e inequality (for instance with the gradient):
\begin{equation} \label{eq:PI-Ap1}
\avgint_Q|f(x)-f_Q|dx\le c \ell(Q) \avgint_Q  g dx, 
\end{equation}
then by inequality \eqref{eq:prop-Ap1-function}, we have that
\begin{eqnarray*}
\avgint_Q|f(x)-f_Q|dx & \leq & c\, \ell(Q) [w]_{A_{p,1}}^{\frac1p} \, \,\Big\| g \Big\|_{L^{p, 1}(Q, \frac{wdx}{w(Q)}  ) }.\\
\end{eqnarray*}
If we let 
$$a(Q):= \ell(Q)\,   \Big\| g \Big\|_{L^{p, 1}(Q, \frac{wdx}{w(Q)}  ) }
$$
%
the proof of Corollary \ref{cor:Lp1(w)-a(Q)-clean} will follow from the next lemma.

\begin{lemma}\label{lem:L-small-Ap1} Let $w$ be a weight and  let $1\leq p<n$. 
Then $a\in SD_p^{n}(w)$.
\end{lemma}

\begin{proof}

First observe that 
$$
\Big\| g \Big\|_{L^{p, 1}(Q, \frac{wdx}{w(Q)}  ) }
=
\Big\| g \Big\|_{L^{p, 1}(Q ) } \, \left(  \frac{1}{w(Q)} \right)^{\frac1p}.
$$
Now, let $\{Q_i\}\in S(L)$.  Then by Holder's inequality and convexity
\begin{eqnarray*}
\sum_{i}a(Q_i)^pw(Q_i) & = & \sum_{i}  |Q_i|^{p/n}   \Big\| g \Big\|^p_{L^{p, 1}(Q_i) }       \\
&\leq & \left (\sum_i |Q_i|\right )^{p/n} \left (\sum_i  \Big\| g \Big\|_{L^{p, 1}(Q_i) }^{p(n/p)'}\right )^\frac{1}{(n/p)'}\\
&\leq & \left (\sum_i |Q_i|\right )^{p/n} \left (\sum_i  \Big\| g \Big\|_{L^{p, 1}(Q_i) }^{p}\right )\\
&\le & \left (\frac{|Q|}{L}\right )^{p/n}  \Big\| g \Big\|_{L^{p, 1}(Q) }^{p}. 
\end{eqnarray*}
where the last estimate follows from the following known lemma which is a consequence of Minkowsky inequality (see \cite[Lemma 2.5]{CHK}  for a more general version). 
\begin{lemma}
$$
\sum_i  \Big\| g \Big\|_{L^{p, 1}(Q_i, wdx ) }^{p} \leq  \Big\| g \Big\|_{L^{p, 1}(Q, wdx ) }^{p}.
$$
\end{lemma}
\end{proof}

Higher order derivative estimates and Sobolev-Poincare estimates can be considered as well but we will not pursue in this direction.

\section{There is no Poincar\'e inequality for \emph{all} \texorpdfstring{$A_\infty$}{Ainfty} weights}\label{sec:noAinfinity}

As mentioned in the introduction,  we prove in this section a negative result which is intimately related to the failure of the Poincar\'e inequality $(p,p)$ 
when $p<1$. We will show that there is no weighted Poincar\'e $(p,p)$ inequality ($p\geq 1$) valid for the class $RH_\infty$ and hence for the class $A_\infty$ since $RH_\infty \subset A_\infty$.  We recall that a weight $w$ belongs to the the class $RH_\infty$ if there is a constant $c$ such that
$$
\sup_Q w \leq c\ \avgint_{Q} w. 
$$
This definition means that $w$ satisfies a reverse H\"older's inequality for any exponent and hence $RH_\infty \subset A_\infty$. We will use the following known lemma.
\begin{lemma}\label{RHinfity}
Let $\lambda>0$ and let $\mu$ be a measure such that $M\mu$ is finite almost everywhere,  then $(M\mu)^{-\lambda} \in RH_{\infty}$ with constant independent of $\mu$. 
\end{lemma}

For the proof we will use that the following well known fact, if $0<\delta<1$,  $(M\mu)^{\delta} \in A_1$ with constant independent of $\mu$ (see \cite{GCRdF}) . Now, 
\begin{eqnarray*}
\sup_Q \, (M\mu)^{-\lambda} & = & \left(\sup_Q\, (M\mu)^{-\delta}\right)^{\frac{\lambda}{\delta}}
=  \left( \frac{1}{ \inf_Q\, (M\mu)^{\delta} } \right)^{\frac{\lambda}{\delta}}\\
& \leq  &\left( \frac{c_{\delta}}{ 
\avgint_{Q} (M\mu)^{\delta}\, dx  } \right)^{\frac{\lambda}{\delta}}
\end{eqnarray*}
if we choose that $0<\delta<1$ using  that $(M\mu)^{\delta} \in A_1$. Now, since 
$$
1\leq  \avgint_{Q} w  \avgint_{Q} w^{-1}
$$
we continue with
$$
\sup_Q \, (M\mu)^{-\lambda} \leq \left( c_{\delta}  \avgint_{Q} (M\mu)^{-\delta}\, dx   \right)^{\frac{\lambda}{\delta}}
\leq c_{\delta,\lambda}  \avgint_{Q} (M\mu)^{-\lambda}\, dx,$$
by Jensen's inequality choosing $\delta$ such that $\lambda>\delta$ if necessary. This finishes the proof that $(M\mu)^{-\lambda} \in RH_{\infty}$ with universal constant.

\begin{proof}[Proof of Theorem \ref{thm:RHinfty}]

Consider some fixed $p_0\in (0,1)$. We will use the modified maximal operator 
\begin{equation*}
M_\varepsilon(f):=M(|f|^\varepsilon)^{1/\varepsilon}, \qquad \varepsilon>0.
\end{equation*}
Then we have, for some $\varepsilon,\alpha>0$ to be chosen later, that
\begin{eqnarray*}
\left (\int_Q|f-a|^{p_0}\right )^\frac{1}{p_0} & = & \left (\int_Q|f-a|^{p_0}M_\varepsilon(\chi_Q|\nabla f|)^{-\alpha p_0}M_\varepsilon(\chi_Q|\nabla f|)^{\alpha p_0}\right )^\frac{1}{p_0}\\
&\le & I.II,
\end{eqnarray*}
by using H\"older inequality with the pair $q=\frac{p}{p_0}>1$ and $q'=\frac{p}{p-p_0}$, where 
\begin{equation*}
I=\left (\int_Q|f-a|^{p}M_\varepsilon(\chi_Q|\nabla f|)^{-\alpha p}\right )^{1/p}
\end{equation*}
and 
\begin{equation*}
II=\left (\int_Q M_\varepsilon(\chi_Q|\nabla f|)^{\alpha p_0\left( \frac{p}{p_0}\right )'}\right )^{\frac{1}{p_0\left( \frac{p}{p_0}\right )'}}.
\end{equation*}
Now, by the Lemma above we have that $M_\varepsilon(\chi_Q|\nabla f|)^{-\alpha p_0}$ belongs to $RH_\infty$ with constant independent of $|\nabla f|$.   Then we can use the hypothesis to control the first factor above by
\begin{eqnarray*}
I & \le & C\ell(Q)\left (\int_Q|\nabla f|^{p}M_\varepsilon(\chi_Q|\nabla f|)^{-\alpha p}\right )^{1/p}\\
& \le &C\ell(Q) \left (\int_Q|\nabla f|^{p-\alpha p}\right )^{1/p}\\
& = & C\ell(Q)\left (\int_Q|\nabla f|^{p_0}\right )^{1/p}
\end{eqnarray*}
choosing $\alpha=\frac{p-p_0}{p}>0$.

For the second factor $II$, note that by the choice of $\alpha$, we have that
$\alpha \left (\frac{p}{p_0}\right )'=\frac{p-p_0}{p} \frac{p}{p-p_0}=1$.
If we consider $\varepsilon=\frac{p_0}{2}\in (0,1)$  then we have that 
\begin{equation*}
II = 
\left (\int_Q M(\chi_Q|\nabla f|^\frac{p_0}{2})^2\right )^{\frac{1}{p_0\left( \frac{p}{p_0}\right )'}}
\leq c\, \left (\int_Q |\nabla f|^{p_0}\right )^{\frac{p-p_0}{p_0p}}
\end{equation*}
by the boundedness of the maximal operator. Therefore, collecting estimates and noting that $\frac{1}{p}+\frac{p-p_0}{p_0p}=\frac{1}{p_0}$, we obtain
\begin{equation*}
\left (\int_Q|f-a|^{p_0}\right )^{1/p_0}\le C\ell(Q) \left (\int_Q |\nabla f|^{p_0}\right )^{\frac{1}{p_0}}.
\end{equation*}

\end{proof}

\section{Appendix:  A general two weight Poincar\'e inequality for the weak norm and representation formula} \label{sec:appendix}

It is very well known that there is a close connection between Poincar\'e inequalities and fractional integral operators. In particular, we will consider the fractional integral operators defined for any $0<\alpha <n $ by
\begin{equation}
I_\alpha(g)(x):=\int_{\mathbb{R}^n} \frac{g(y)}{|x-y|^{n-\alpha}}dy
\end{equation}

We will use the following lemma regarding the normability of the space $L^{p,\infty}$ from \cite{GrafakosCF}. Let $\mu$ be any Radon measure on $\mathbb{R}^n$. The weak ``norm'' with respect to $\mu$  is defined by
\begin{equation*}
\|f\|_{L^{p,\infty}_{\mu}}:=\sup_{\lambda>0}\lambda \mu(\{x\in \mathbb{R}^n: |f(x)|>\lambda\})^\frac{1}{p}.
\end{equation*}
The space $L^{p,\infty}_{\mu}$ is the set of functions such that the quantity above is finite.

\begin{lemma}\label{lem:nomability-WeakLp}
Let $p>1$. Define, for any $f$ in the weak space $L_{\mu}^{p,\infty}$, the norm
\begin{equation}\label{eq:nomability-WeakLp}
		\vertiii{f}_{L_{\mu}^{p,\infty}}:=\sup_{0<\mu(E)<\infty} \mu(E)^{\frac{1}{p}-1}\int_E |f(x)|\ d\mu
\end{equation}
Then we have that
\begin{equation}
\|f\|_{L_{\mu}^{p,\infty}}\le \vertiii{f}_{L_{\mu}^{p,\infty}}\le p'\|f\|_{L_{\mu}^{p,\infty}}
\end{equation}
\end{lemma}

\subsection{The truncation method or ``\emph{weak implies strong}"} \label{truncation}

We also include here a general ``\emph{weak implies strong}'' argument valid in our context of Poincar\'e type inequalities. The following lemma provides an argument to obtain strong estimates from weak type inequalities when the right hand side involves a gradient. It seems that goes to the work of Mazja, however it can be explicitly found in \cite{LN} in the context of $\mathbb{R}^n$  and in \cite{SW} in the context of Poincar\'e inequalities.
We refer to \cite{Ha} for more information about it.

\begin{lemma}\label{lem:weak-strong}
Let $g$ be any nonnegative Lipschitz function. Suppose that for some $p>1$ there is a weak $(1,p)$-type estimate for a pair of measures $\mu,\nu$ of the form:
\begin{equation*}
\sup_{t>0}\, t \, \mu(\{x\in \mathbb{R}^n: g(x)>t\})^{1/p}\lesssim \int_{\mathbb{R}^n}|\nabla g(x)| d\nu
\end{equation*}
Then the strong estimate also holds, namely
\begin{equation*}
\|g\|_{L^p_\mu}\lesssim \int_{\mathbb{R}^n}|\nabla g(x)| d\nu
\end{equation*}
\end{lemma}
\begin{proof}
Define, for any real parameter $\lambda\in \mathbb R$, the truncation $T_\lambda(g)$ as follows:
\begin{equation*}
T_\lambda(g)(x):=
\left \{\begin{array}{cc}
0& \text{ if } g(x)\le \lambda\\
g(x) - \lambda & \text{ if } \lambda \le g(x)\le 2\lambda\\
\lambda & \text{ if } g(x)\ge 2\lambda\\
\end{array}
\right .
\end{equation*}
Also define for each $k\in \mathbb Z$ the set $G_k:=\{x\in \mathbb{R}^n: 2^k< |g(x)|\le 2^{k+1}\}$.
We have that, for all $x\in G_{k+1}$, $T_{2^k}(|g|)(x)=2^k$ and $\text{sop}\nabla(T_{2^k}(g))\subset G_k$. We proceed as follows:
\begin{eqnarray*}
\left (\int_{\mathbb{R}^n}g(x) d\mu\right )^{1/p} &\le &\sum_{k=-\infty}^{k=\infty} 2^k\mu(G_{k+1})^{1/p}\\
& \le  & \sum_{k=-\infty}^{k=\infty} 2^k\mu(x\in \mathbb{R}^n : T_{2^{k}}(g)(x)>2^{k-1})^{1/p} \\
& \lesssim  &\sum_{k=-\infty}^{k=\infty}  \int_{\mathbb{R}^n}|\nabla T_{2^{k}}(g)(x)| d\nu\\
& \le  & \sum_{k=-\infty}^{k=\infty}  \int_{G_k}|\nabla g (x)| d\nu\\
& \le  & \int_{\mathbb{R}^n}|\nabla g (x)| d\nu
\end{eqnarray*}

\end{proof}

The last result in this section is attached to the  $(1,1)$ Poincar\'e inequality for $L^1$:
\begin{equation}\label{eq:Poincare-L1}
\frac{1}{|Q|}\int_Q|f(x)-f_Q|dx\le C \ell(Q)\frac{1}{|Q|}\int_Q |\nabla f(x)| dx.
\end{equation}
This result is well known and can be found in many places: \cite{WZ}, \cite{Saloff-Coste-LMS-LN}. This estimate is also valid replacing the cube $Q$ by any convex set $\Omega\subset \mathbb{R}^n$ where the natural substitute for the sidelength of the cube is the diameter.  The proof of this result is well known but it has been shown in \cite{AD04} the very interesting fact that $\dfrac{1}{2}\text{diam}(\Omega)$ is the best constant. Interesting extensions of this result can be found in \cite{H} and in \cite{DD}.

We show next that  \eqref{eq:Poincare-L1}  encodes an intrinsic information by showing that it is equivalent to the following statements connecting $(p,p)$, weak or strong. Poincar\'e type inequalities, pointwise inequalities involving fractional operators and corresponding two-weighted estimates.

\begin{theorem}\label{thm:equiv-weak-strong-1n'-pointiwise} The following are equivalent. 

1) The following Poincar\'e inequality holds 

\begin{equation}\label{eq:Poincare-L1-fractional}
\avgint_Q|f(x)-f_Q|dx\le C \ell(Q)\avgint_Q |\nabla f(x)| dx,
\end{equation}
for any cube $Q$ with a constant $C$ not depending on the cube.

2) The following pointwise estimate holds
\begin{equation}\label{eq:pointwise-I1}
|f(x)-f_Q|\le C_n I_1(|\nabla f|\chi_Q)(x) \qquad \forall x\in Q,
\end{equation}
where $C_n$ depends on the dimension $n$.

3) If  $\mu$  is any Radon measure on $\mathbb{R}^n$, $n\ge2$,  $Q$ a cube, and $f$ a Lipschitz function, then
\begin{equation}\label{eq:weak-1n'}
\|(f-f_Q)\chi_Q\|_{L_{\mu}^{n',\infty}} \le C \int_Q |\nabla f(x)| (M\mu(x))^\frac{1}{n'}\ dx.
\end{equation} 

4) If  $\mu$  is any Radon measure on $\mathbb{R}^n$, $n\ge2$,  $Q$ a cube, and $f$  a Lipschitz function, then 

\begin{equation}\label{eq:P(1,n')-mu-local-avg} 
\left (\int_Q |f(x)-f_Q|^{n'}\ d\mu\right )^\frac{1}{n'}\le C \int_Q |\nabla f(y)| (M\mu(y))^\frac{1}{n'}\ dy.
\end{equation}

\end{theorem}

\begin{corollary}\label{cor:P(1,n')-mu-local-avg}

Let  $\mu$  be any Radon measure on $\mathbb{R}^n$, $n\ge2$. Then there is a dimensional constant $C$ such that for an Lipschitz function $f$ with compact support, 
\begin{equation*}
\left (\int_{\mathbb{R}^n}  |f(x)|^{n'}\ d\mu\right )^\frac{1}{n'}\leq C\, \int_{\mathbb{R}^n}    |\nabla f(y)| (M\mu(y))^\frac{1}{n'}\ dy
\end{equation*} 
\end{corollary}

The proof  follows from \eqref{eq:P(1,n')-mu-local-avg}   letting $\ell(Q)\to \infty$ using that $f_Q \to 0$.

\begin{remark} \label{fractional}

There is a corresponding fractional version replacing $1)$ above by 
$$
\avgint_Q|f(x)-f_Q|dx\le C \ell(Q)^\alpha\avgint_Q |\nabla f(x)| dx.
$$
The corresponding fractional operator is $I_\alpha$ in $2)$ instead of $I_1$ and the weak norm in $3)$ is in $L^{\frac{n}{n-\alpha},\infty}$ in $3)$. The implication $3)\Rightarrow 4)$ is a consequence of Lemma \ref{lem:weak-strong} and therefore is still valid.

\end{remark}

\begin{proof}

1) $\Longrightarrow $  2). 

We will adapt the main ideas from \cite{FLW} (see also \cite{LuPerez02}). 
We will derive  \eqref{eq:pointwise-I1} from \eqref{eq:Poincare-L1-fractional} by using the Lebesgue differentiation theorem. Let $x\in Q$. Then there is a chain $\{Q_k\}_{k\ge 1}$ of nested \emph{dyadic} subcubes of $Q$ such that  $Q_1=Q$, $Q_{k+1}\subset Q_k$ for all $k\ge 1$ and 
\begin{equation*}
\{x\}=\bigcap_{k\ge 1} Q_k
\end{equation*}
Let $f_{Q_k}$ be the average of $f$ over the cube $Q_k$. Then by the Lebesgue differentiation theorem, there exists a null set $N$ such that for all $x\in E:=Q\setminus N$ we have that 
\begin{equation*}
|f(x)-f_Q|=|\lim_{k\to\infty} f_{Q_k} - f_Q|=|\sum_{k\ge 1} f_{Q_{k+1}}- f_{Q_k}|
\end{equation*}
Now, using the dyadic structure of the chain, we obtain that 
\begin{eqnarray*}
|\sum_{k\ge 1} f_{Q_{k+1}}- f_{Q_k}| & \le & \sum_k \frac{1}{|Q_{k+1}|}\int_{Q_{k+1}}|f_{Q_k}-f|\\
& \le & 2^n \sum_k \frac{1}{|Q_{k}|}\int_{Q_{k}}|f_{Q_k}-f|\\
& \le & C2^n \sum_k \frac{\ell(Q_k)}{|Q_{k}|}\int_{Q_{k}}|\nabla f(y)| dy\\
& = & C2^n \int_{Q}|\nabla f(y)|\sum_k \ell(Q_k)^{1-n}\chi_{Q_k}(y) dy
\end{eqnarray*}

Note that the immediate estimate $|x-y|\le \sqrt{n}\ell(Q_k)$ produces an extra unwanted $\log$ factor when summing the series. We instead proceed as follows. Let $0<\eta< n-\alpha$. Then we have that 
\begin{eqnarray*}
\sum_k \ell(Q_k)^{1-n}\chi_{Q_k}(y) &\le & \frac{C_n}{|x-y|^{n-1-\eta}}\sum_k \ell(Q_k)^{-\eta}\chi_{Q_k}(y)\\
&\le &\frac{C_n}{|x-y|^{n-1-\eta}\ell(Q)^\eta}\sum_{k=0}^{k_0} 2^{k\eta}
\end{eqnarray*}
for $k_0=\min\{j\in \mathbb{N}: 2^{j}>\sqrt{n}\frac{\ell(Q)}{|x-y|}\}$. Then we obtain that

\begin{equation*}
\sum_k \ell(Q_k)^{1-n}\chi_{Q_k}(y) 
\le C_n
\frac{2^{\eta k_0}}{|x-y|^{n-1-\eta}\ell(Q)^\eta}
\le C_n \frac{1}{|x-y|^{n-1}}
\end{equation*}
Collecting all previous estimates, we conclude with the proof of the desired inequality
\begin{equation*}
|f(x)-f_Q|\le C_n \int_Q \frac{(|\nabla f|\chi_Q)(y)}{|x-y|^{n-1}}dy \qquad \forall x\in Q
\end{equation*}

2) $\Longrightarrow $  3).

We can compute the weak norm by using Lemma \ref{lem:nomability-WeakLp}. More precisely, we can apply Fubini-Tonelli's theorem and \eqref{eq:nomability-WeakLp} to obtain

\begin{eqnarray*}
\|(f-f_Q)\chi_Q\|_{L_{\mu}^{n',\infty}} &\lesssim & \|I_1(|\nabla f|\chi_Q)\chi_Q\|_{L_{\mu}^{n',\infty}}\\
& \lesssim & \vertiii{\int_{\mathbb{R}^n}\frac{|\nabla f(y)|\chi_Q(y)}{|\cdot-y|^{n-1}}\chi_Q(\cdot)\ dy}_{L_{\mu}^{n',\infty}}\\
& \lesssim & \int_{\mathbb{R}^n} \|K(\cdot,y)\chi_Q(\cdot)\|_{L_{\mu}^{n',\infty}}|\nabla f(y)| \chi_Q(y) \ dy\\
& \lesssim & \int_{Q} \|K(\cdot,y)\chi_Q(\cdot)\|_{L_{\mu}^{n',\infty}}|\nabla f(y)|\ dy.
\end{eqnarray*}

Now we estimate the inner norm of the kernel $K(x,y)=\frac{1}{|x-y|^{n-1}}$, $x,y\in Q$. By definition of the weak norm, we have
\begin{eqnarray*}
\|K(\cdot,y)\chi_Q\|_{L_{\mu}^{n',\infty}}&=&\sup_{t>0} \left ( t^{n'}\mu\left (x\in Q: \frac{1}{|x-y|^{n-1}}>t\right )\right )^\frac{1}{n'}\\
& \lesssim & \sup_{r>0} \left ( r^{-n}\mu\left (x\in Q: |x-y|<r \right )\right )^\frac{1}{n'}\\
& \lesssim & \sup_{r>0} \left ( |B(y,r)|^{-1} \mu\left ( B(y,r) \right )\right )^\frac{1}{n'}\\
& \lesssim & (M^c\mu (y))^\frac{1}{n'}
\end{eqnarray*}
Recall that $M^c$ denotes the centered maximal function.
Therefore, collecting all estimates, we obtain that

\begin{equation*}
\|(f-f_Q)\chi_Q\|_{L_{\mu}^{n',\infty}}\le C \int_Q |\nabla f(y)| (M\mu(y))^\frac{1}{n'}\ dy
\end{equation*} 

3) $\Longrightarrow $  4).  

This follows from Lemma \ref{lem:weak-strong}.

4) $\Longrightarrow $  1).  

This follows by considering as measure $\mu$ the Lebesgue measure. 

\end{proof}

\begin{remark} \label{caracterizacion}
We remark that we could obtain 2) directly from 4) by evaluating estimate \eqref{eq:P(1,n')-mu-local-avg} in a Dirac measure. In fact,  for any $x_0\in \mathbb{R}^n$
we let  $\mu=\delta_{x_0}$ be the Dirac measure concentrated at $x_0$. Then an easy computation of the maximal function $M\mu  $
using that $M\mu(x) \approx |x-x_0|^{-n}$ yields that for any $Q\ni x_0$,
\begin{equation*}
\int_Q |\nabla f(y)| (M\mu(y))^\frac{1}{n'}\ dy\le I_1(|\nabla f|\chi_Q)(x_0).
\end{equation*} 
On the other hand, we also have that
\begin{equation*}
\|f-f_Q\|_{L_{\mu}^{n'}}=|f(x_0)-f_Q|.
\end{equation*} 
Then, using \eqref{eq:weak-1n'} we obtain that 
\begin{equation*}
|f(x_0)-f_Q| = \|f-f_Q\|_{L_{\mu}^{n'}}\le \int_Q |\nabla f(y)| (M\mu(y))^\frac{1}{n'}\ dy\le I_1(|\nabla f|\chi_Q)(x_0),
\end{equation*}
which is exactly \eqref{eq:pointwise-I1}.
\end{remark}

\subsection{Another proof of one weight and two weight Poincar\'e inequalities}
We present here a ``classical'' proof one and two weights Poincar\'e inequalities. The proofs follow from the representation in terms of the fractional integral from \eqref{eq:pointwise-I1}. 

\begin{proposition}\label{pro:P(p,p)-I1vsM}
Let $w\in A_p$. Then for any Lipschitz function $f$ and for any cube $Q$ we have that
\begin{equation*}
\left (\frac{1}{w(Q)}\int_Q |f-f_Q|^p\, w\ dx\right )^{1/p}\le [w]^{\frac{1}{p}}_{A_p}\ell(Q)\left (\frac{1}{w(Q)}\int_Q|\nabla f|^{p} \, w \ dx\right )^{1/p}
\end{equation*}
\end{proposition}

\begin{proof}
The idea is to control the fractional integral $I_1$ by the maximal function at the level of weak norms to obtain the precise weighted estimate. Then, the weak implies strong strategy concludes the proof.
More precisely, we have from \eqref{eq:pointwise-I1} that 
$$
|f(x)-f_Q|\le C_n I_1(|\nabla f|\chi_Q)(x) \qquad \forall x\in Q,
$$
where we assume that the cube $Q$ is of the form $Q=Q(z,r)$, a cube centered at the point $z$ and sidelength $\ell(Q)=2r$. Then, for any $x\in Q$, we have that $Q\subset Q(x,2r)$. We now decompose the cube $Q(x,2r)$ in annular regions of the form $Q_k=Q(x,2^{-k+1}r)\setminus Q(x,2^{-k}r)$, $k\ge 0$. Then,
\begin{eqnarray*}
\int_Q \frac{|\nabla f|(y)}{|x-y|^{n-1}} dy& \le & \int_{Q(x,2r)} \frac{|\nabla f|(y)}{|x-y|^{n-1}}dy\\
& \le & \sum_k \int_{Q_k}\frac{|\nabla f|(y)}{|x-y|^{n-1}}dy\\
& \le & C \sum_k \frac{2^{-k}r}{(2^kr)^{n}}\int_{Q_k}|\nabla f|(y)dy\\
&\le & C M(|\nabla f|)(x)\sum_k 2^{-k}r\\
& \le & C \ell (Q) M(|\nabla f|)(x).
\end{eqnarray*}
Therefore we conclude that 
\begin{equation}\label{eq:I1-vs-M}
I_1(|\nabla f|\chi_Q)(x)\le C \ell (Q) M(|\nabla f|)(x)
\end{equation}
for any $x\in Q$. Now we compute the weak $L^p$ norm to obtain
\begin{eqnarray*}
\|f-f_Q\|_{L^{p,\infty}_{Q,w}} & \le & \|  I_1(|\nabla f|\chi_Q)\|_{L^{p,\infty}_{Q,w}}\\
& \le & \ell(Q) \|  M(|\nabla f|)\|_{L^{p,\infty}_{Q,w}}\\
& \le & \ell(Q) [w]_{A_p}^{\frac{1}{p}}\|\nabla f\|_{L_Q^p(w)}
\end{eqnarray*}
by the known weighted estimate for the weak norm for the maximal function. The \emph{weak implies strong} argument finishes the proof.
\end{proof}

We have the following analogue of Proposition \ref{pro:P(p,p)-I1vsM} for pairs of weights.

\begin{proposition}\label{pro:two-weights}
Let $(u,v) \in A_p$. Then the following Poincar\'e $(p,p)$ inequality holds for any Lipschitz function $f$.
$$
\left( \int_Q|f-f_{Q}|^{p}u\right )^\frac{1}{p}\le C [u,v]^{1/p}_{A_p}\, \ell(Q)  \left (\int_Q |\nabla f|^p v\right )^\frac{1}{p}.
$$
\end{proposition}

\begin{proof}
The proof is the same as in Proposition \ref{pro:P(p,p)-I1vsM}. We can use the pointwise inequality 
\eqref{eq:pointwise-I1} together with \eqref{eq:I1-vs-M} to obtain
\begin{equation*}
|f(x)-f_Q|\le C_n\ell(Q) M(|\nabla f|\chi_Q)(x) \qquad x\in Q.
\end{equation*}
Taking weighted $L^p$ norms we obtain that
\begin{eqnarray*}
\left\| f-f_{Q}\right \|_{L^{p,\infty}_{Q,u}}
& \le & C_n\ell(Q) \left\| M(|\nabla f|\chi_Q)\right \|_{L^{p,\infty}_{Q,u}}\\
& \le  & \|M\|_{L^p(u)\to L^{p,\infty}(v)}\left( \int_Q|\nabla f|^{p}\,v\right )^\frac{1}{p}\\
& \le &  [u,v]^{1/p}_{A_p}\left( \int_Q|\nabla f|^{p} \,v\right )^\frac{1}{p},
\end{eqnarray*}
since by \eqref{charactAp-two-weights} 
$$\|M\|_{L^p(v)\to L^{p,\infty}(u)} \approx  [u,v]^{1/p}_{A_p}.
$$
Now, the \emph{weak implies strong } argument finishes the proof.
\end{proof}

By assuming an extra condition on the weight $u$, we have another result.

\begin{corollary}
Let $(u,v)\in A_p$,  $u\in A_p$,  $1 \leq p<n$, and $p_w^*$ satisfying 
\begin{equation}
\frac{1}{p} -\frac{1}{ p_w^* }=\frac{1}{n(p+\log [w]_{A_p})}.
\end{equation}
Then the following two-weight Poincar\'e-Sobolev $(p^*,p)$,inequality holds
$$
\left (\frac{1}{u(Q)}\int_Q |f-f_Q|^{p_w^*}\, u\ dx\right )^{\frac{1}{p_w^*}} \leq [u,v]^{\frac{1}{p}}_{A_p}\ell(Q)\left (\frac{1}{u(Q)}\int_Q|\nabla f|^{p} \, v \ dx\right )^{\frac{1}{p}}.
$$
\end{corollary}
\begin{proof}
The result follows from Lemma \ref{lem:smallAq} and Theorem \ref{thm:Lp(w)-a(Q)-clean} as in the proof of Corollary \ref{cor:ptimes-Aq}.
\end{proof}

\section{Acknowledgments}
This work started during a visit of the second author to the Basque Center for Applied Mathematics (BCAM) in Bilbao, Spain. The second author is deeply grateful to all the staff for their hospitality. Both authors thank Ricardo Dur\'an from UBA for some ideas regarding the proof of Theorem \ref{thm:PoincareSobolev-2weights}.

\bibliographystyle{amsalpha}

\end{document}